\documentclass[a4paper, 12pt]{amsart}
\usepackage{amssymb}
\usepackage{amsthm}
\usepackage{amsmath}
\usepackage{cite}
\usepackage{bbm}
\usepackage{stmaryrd} 
\usepackage{url} 
\usepackage{lineno}
\usepackage[margin=2.3cm]{geometry}
\usepackage{tikz}
\usetikzlibrary{arrows}
\usepackage{mathabx}

\usepackage[all,cmtip,color]{xy} 
\usepackage{float}

\usepackage[textsize=footnotesize]{todonotes}

\usepackage{hyperref}
\definecolor{dblue}{rgb}{0,0,0.70}
\hypersetup{
	unicode=true,
	colorlinks=true,
	citecolor=dblue,
	linkcolor=dblue,
	anchorcolor=dblue, 
	urlcolor   = dblue
}

\newtheorem{lemma}{Lemma}[section]
\newtheorem*{lemma*}{Lemma}
\newtheorem{theorem}[lemma]{Theorem}
\newtheorem{corollary}[lemma]{Corollary}
\newtheorem{proposition}[lemma]{Proposition}
\newtheorem{claim}[lemma]{Claim}

\newtheorem*{question*}{Question}

\theoremstyle{definition}
\newtheorem{definition}[lemma]{Definition}
\newtheorem*{definition*}{Definition}

\newtheorem{question}{Question}

\DeclareMathOperator{\On}{On}

\DeclareMathOperator{\Card}{Card}
\DeclareMathOperator{\Reg}{Reg}

\DeclareMathOperator{\id}{id}
\DeclareMathOperator{\cf}{cf}

\DeclareMathOperator{\ZFC}{\axiomft{ZFC}}

\DeclareMathOperator{\LST}{LST}
\DeclareMathOperator{\NM}{NM}
\DeclareMathOperator{\col}{Col}
\DeclareMathOperator{\height}{ht}

\newcommand{\PP}{\mathbb{P}}
\newcommand{\QQ}{\mathbb{Q}}

\newcommand{\axiomft}[1]{\mathsf{#1}}

\newcommand{\forces}{\mathrel{\Vdash}}

\newcommand{\cof}{\mathrm{Cof}}

\DeclareMathOperator{\Succ}{Succ}

\DeclareMathAlphabet{\mathbbold}{U}{bbold}{m}{n}

\usepackage{enumitem}

\newenvironment{enumerate-(a)}{\begin{enumerate}[label={\upshape (\alph*)}, leftmargin=2pc]}{\end{enumerate}}
\newenvironment{enumerate-(1)}{\begin{enumerate}[label={\upshape (\arabic*)}, leftmargin=2pc]}{\end{enumerate}}
\newenvironment{enumerate-(i)}{\begin{enumerate}[label={\upshape (\roman*)}, leftmargin=2pc]}{\end{enumerate}}

\title{LST Numbers for $Q^\text{e.c.}$ and $I$ style quantifiers}
\author{Christopher Henney-Turner}
\address{Institute of Mathematics of the Polish Academy of Science (IMPAN)
18 Antoniego Abrahama
81-825 Sopot
Poland}
\email{cturner@impan.pl}

\date{\today}

\thanks{This paper is adapted from part of the author's Ph.D. thesis, and he is grateful to his supervisor Philip Welch for his invaluable guidance. He would also like to thank Andrew Brooke-Taylor and Jonathan Osinski for pointing out various glitches in a previous version of the paper. The author was supported in his research by EPSRC grant EP/Student 2123452 and a scholarship from the Heilbronn Institute, and by internal grants from the University of Bristol and the Polish Academy of Science. For the purpose of open access, the author has applied a ‘Creative Commons Attribution' (CC BY) public copyright licence to any Author Accepted Manuscript (AAM) version arising from this submission.} 

\begin{document}
	
	\begin{abstract}
		We introduce two schemes of quantifiers analogous to the H\"artig quantifier $I$ and the Equal Cofinality quantifier $Q^\text{e.c.}$, which tell us about regular cardinals of small Cantor-Bendixson rank. We examine how the L\"owenheim-Skolem-Tarski numbers of these quantifiers interact with one another, and with those of $I$ and $Q^{\text{e.c.}}$. We then find the exact lower bound for each of the $\LST$ numbers, assuming the consistency of supercompacts.
	\end{abstract}
	
\maketitle
	
\section{Introduction}



The L\"owenheim-Skolem theorem famously says that for any first-order language $\mathcal{L}$, any first order $\mathcal{L}$ structure contains an elementary substructure of size less than $\max(\lvert \mathcal{L}\rvert,\omega_1)$. The concept of the L\"owenheim-Skolem-Tarski number generalises this to logics with non-first order elements. The $\LST$ number of a logic is the smallest cardinal $\kappa$ such that every structure contains a substructure of size less than $\kappa$.  The L\"owenheim-Skolem theorem tells us that the $\LST$ number of a first order logic is $\omega_1$, and thus we can think of the $\LST$ number as quantifying how close a chosen logic is to being first order. If we define (non-first order) logical symbols which express a particular concept, then calculating the $\LST$ number of the corresponding logic tells us the complexity of that concept.

In \cite{magidorVaananan}, Magidor and V\"a\"an\"anen investigate $\LST$ numbers for two \textit{generalised quantifiers} which express concepts from set theory: the H\"artig quantifier $I$ (which expresses cardinality, by allowing us to distinguish which of the ordinals are in the class $\Card$ of cardinals), and the equal cofinality quantifier $Q^{\text{e.c.}}$ (which expresses regularity, by letting us distinguish which cardinals are in $\Reg$, the class of regular cardinals). In both cases, the lower bound turns out to be precisely the smallest sort of large cardinal which is not identified by the quantifier:

\begin{theorem}[Magidor,V\"a\"an\"anen]\cite[7, 20 \& 21]{magidorVaananan}
	If it exists, then $\LST(I)$ is at least the first inaccessible cardinal. Similarly, $\LST(I,Q^{\text{e.c.}})$ is at least the first Mahlo cardinal, if it exists. Moreover, if it is consistent that a supercompact cardinal exists then it is also consistent that either one of these $\LST$ numbers is exactly equal to the bound given.
\end{theorem}

There is a large gap in complexity between a logic which identifies only the cardinals, and one which identifies all of the regulars; as we can see from the fact that the lower bound for $\LST(I,Q^{\text{e.c.}})$ requires an entirely different kind of large cardinal to the lower bound for $\LST(I)$. In fact, the lower bound for $\LST(I)$ is the first inaccessible; that is, the first limit point of $\Reg$. Recalling that $\Card$ is the closure of the class of successors of $\Reg$, this suggests that from the perspective of $\LST$ numbers, the complexity of a class of (some of the) regular cardinals is based on the class's \textit{Cantor-Bendixson rank} -- that is, whether the class contains limits of its elements, and limits of limits, and limits of limits of limits, etc.



In this paper, we test this by finding exact lower bounds for the $\LST$ numbers of logics related to a family of classes $\Reg_{<\alpha}\subset \Reg$ (for $\alpha \in \On$). $\Reg_{<\alpha}$ is the canonical subclass of $\Reg$ of Cantor-Bendixson rank $\alpha$, consisting of all the elements of $\Reg$ of rank less than $\alpha$.

\begin{definition}\label{definition_Reg_alpha}
	For $\alpha\in \On$, we recursively define the classes $\Reg_\alpha$, $\Reg_{\geq \alpha}$ and $\Reg_{<\alpha}$ as follows:
	\begin{itemize}
	\item $\Reg_{<\alpha}:= \bigcup_{\beta<\alpha} \Reg_\beta$
	\item $\Reg_{\geq \alpha}:=\Reg\setminus \Reg_{<\alpha}$
	\item $\Reg_\alpha$ is the class of successors of the club generated by $\Reg_{\geq\alpha}$
\end{itemize}
\end{definition}

So $\Reg_0$ is the class of successor cardinals, $\Reg_1$ is the class of simple inaccessibles, and so on. $\Reg_{<\alpha}$ is a natural choice for a class of Cantor-Bendixson rank $\alpha$: it is definable from $\Reg$ in a canonical way, it includes all uncountable regular cardinals of smaller complexity, and it is ``almost maximal'' among the classes of rank $\alpha$ -- if $\Reg_{<\alpha}\subset C \subset \Reg$ and $C$ has Cantor Bendixson rank $\alpha$, then $C\setminus\Reg_{<\alpha}$ is bounded. Also, note that $\Reg_0$ is the class of successor cardinals (and thus we can easily define the classes $\Reg_0$ and $\Card$ from one another in any reasonable set theoretic language); and that $\Reg_{<\On}:=\bigcup_{\alpha\in \On}\Reg_{<\alpha} = \Reg \setminus \{\omega\}$. The family of $\Reg_{<\alpha}$ are therefore natural intermediate classes between $\Card$ and $\Reg$, when we are working in situations where limits seem to constitute an increase in complexity.

Our first task is to define formal quantifiers which express these classes in the same way that $I$ and $Q^{\text{e.c.}}$ express $\Card$ and $\Reg$ respectively. This is more fiddly than we might expect, since for $\beta<\alpha$ we want the symbol we come up with for $\Reg_{<\beta}$ to be definable from the symbol for $\Reg_{<\alpha}$. This means we have to encode the Cantor-Bendixson ranking into our symbols in a canonical way. We do this by asking for an additional input to the quantifier, which predicts what $\Reg_{\gamma}$ we will find our regular cardinal in. This means that the quantifier we end up with for $\Reg_{<\On}$ encodes more information than $Q^{\text{e.c.}}$. However, we will show that their $\LST$ numbers are always equal, and therefore they are equivalent as far as our investigation into $\LST$ numbers is concerned.

We will actually look at two different schemes of quantifiers. The first, $Q^\alpha$, is essentially a restriction of $Q^{\text{e.c.}}$ (or more precisely, the modified version of $Q^{\text{e.c.}}$ discussed above) to orders with cofinality in $\Reg_{<\alpha}$. The second, $R^\alpha$, is adapted from $I$ and tells us directly whether a set has cardinality in $\Reg_{<\alpha}$. We will compare the $\LST$ numbers of both of these quantifiers for different values of $\alpha$ (with $I$ added to the logic, as is standard) and will find lower bounds for each of them. We will then find a universe in which both of these lower bounds are met. The conclusion we shall reach is:

\begin{theorem}
	Let $\alpha\in \On\setminus\{0\}$. If $\alpha$ is below the first Mahlo cardinal, then $\LST(I,Q^\alpha)$ is at least the first element of $\Reg_\alpha$. If $\alpha$ is below the first hyperinaccessible, then the same is true of $\LST(I,R^\alpha)$. Moreover, for any $\alpha$, if $\kappa>\alpha$ is supercompact and $V\vDash \text{GCH}$ then there is a generic extension in which $\alpha$ is below both of these limits (in fact, countable) and in which $\LST(I,Q^\alpha)=\LST(I,R^\alpha)=\kappa =\min \Reg_\alpha$, so this lower bound is optimal.
\end{theorem}



\section{Preliminaries}

Throughout this paper, we use the concept of a ``generalised quantifier'' developed by Lindstr\"om in \cite{Lindstrom}. The following is a slight simplification of Lindstr\"om's definition, which is equivalent to the full definition for all the quantifiers we will be looking at.

\begin{definition}
	A \textit{type $(n_1,\ldots,n_k)$ generalised quantifier} is a formal symbol $Q$ together with a map $V\rightarrow V$ which takes any set $A$ to some $Q_A \subset \mathcal{P}(A^{n_1})\times \ldots \times \mathcal{P}(A^{n_k})$. If $\mathcal{L}$ is a language (i.e. a set of symbols) which contains $Q$ and some first order symbols, then the (syntactically valid) $\mathcal{L}$ formulas are defined recursively. The usual (recursive) definition is used for atomic formulas, $\vee, \wedge, \neg, \forall$ and $\exists$; but we also allow formulas of the form
	$$Q v_{1,1}, \ldots, v_{1,n_1}, v_{2,1}, \ldots,v_{k,n_k}(\varphi_1 \ldots \varphi_k)$$
	where the $v_{i,j}$ are (not necessarily distinct) variable symbols, the $\varphi_i$ are syntactically valid formulas, and $v_{i,j}$ is a free variable of $\varphi_i$. For ease of notation, let us write $\varphi_i = \varphi_i(v_{i,1}\ldots,v_{i,n_i})$ (supressing any other free variables). If $\mathcal{A}$ is an $\mathcal{L}\setminus\{Q\}$ structure with domain $A$, we extend it to an $\mathcal{L}$ structure by interpreting
	$$ Q v_{1,1}, \ldots, v_{1,n_1}, v_{2,1}, \ldots,v_{k,n_k}(\varphi_1 \ldots \varphi_k)[a]$$
	as true (for a given assignment $a$)\footnote{In order to align with the first-order notation used in set theory when possible, we will commit a minor abuse of notation by wrapping assignments of variables in square brackets when they follow a generalised quantifier, but using round brackets otherwise.} if and only if
	$$(S_1, \ldots, S_k)\in Q_A$$
	where
	$$S_i:=\{(x_1,\ldots,x_{n_i})\in A^{n_i} : \mathcal{A}\vDash\varphi_i(x_1,\ldots,x_{n_1})[a]\}$$
	
	We also extend this definition in the obvious way to languages $\mathcal{L}$ containing more than one generalised quantifier.
\end{definition}

Note that this definition allows us to treat formulas involving $Q$ like any other formulas; for example, we can nest quantifiers inside each other, combine two formulas with $Q$ using $\vee$ or $\wedge$, etc.

The two well-known quantifiers mentioned in the introduction are Lindstr\"om quantifiers:

\begin{definition}
	The H\"artig quantifier $I$ is a type $(1,1)$ quantifier. For any set $A$,
	$$I_A=\{(A,B)\in \mathcal{P}(A) \times \mathcal{P}(A): \lvert A \rvert = \lvert B \rvert\}$$
\end{definition}

Note that the evaluation of cardinality in the above definition takes place in $V$. So

$$I v_{1,1}, v{2,1} (\varphi_1,\varphi_2)[a]$$
is true if and only if the subsets of $\mathcal{A}$ defined by $\varphi_1(v_{1,1})$ and by $\varphi_2(v_{2,1})$ have the same $V$ cardinality.

\begin{definition}
	The equal cofinality quantifier $Q^{\text{e.c.}}$ is a type $(2,2)$ quantifier. For any set $A$,
	$$Q^{\text{e.c.}}_A = \{(A,B) \in \mathcal{P}(A^2) \times \mathcal{P}(A^2): A \text{ and } B \text{ are linear orders with the same cofinality}\}$$
\end{definition}

If $\mathcal{L}$ is a language involving quantifiers and first order symbols, we define elementarity between $\mathcal{L}$ structures in the obvious way. The final standard concept to define is the $\LST$ number.

\begin{definition}
	Let $Q_0,\ldots,Q_n$ be generalised quantifiers. The \textit{L\"owenheim-Skolem-Tarski} number $\LST(Q_0,\ldots,Q_n)$ of $Q_0,\ldots,Q_n$ is the least infinite cardinal $\kappa$ such that the following holds: For every first order vocabulary $\mathcal{L}$ of cardinality less than $\kappa$, for every $\mathcal{L}$ structure $\mathcal{A}$, there is an $\mathcal{L}\cup\{Q_0,\ldots,Q_n\}$ elementary substructure $\mathcal{B}\leq \mathcal{A}$ of size less than $\kappa$. (By this we technically mean that the extension of $\mathcal{B}$ to the language generated by the vocabulary $\mathcal{L}\cup \{Q_0,\ldots,Q_n\}$ is an elementary substructure of the similar extension of $\mathcal{A}$. To avoid repeating these technical details, we will use this phrasing without comment from this point on.)

	If no such $\kappa$ exists, then we say that $\LST(Q_0,\ldots,Q_n)=\infty$ or that it does not exist.
\end{definition}

\section{The new quantifiers}

It is now time to define the quantifiers $Q^\alpha$ and $R^\alpha$ we're going to be looking at. Both are intended to characterise $\Reg_{<\alpha}$. We define $Q^\alpha$ first. It is similar to $Q^{\text{e.c.}}$, but in order to restrict it to $\Reg_{<\alpha}$ in a natural way, we add an extra requirement: as well as requiring that the two linear orders have the same cofinality, we ask for an auxiliary formula deciding the Cantor-Bendixson rank of that cofinality.

\begin{definition}
	Let $\alpha \leq \On$. The quantifier $Q^\alpha$ is type $(2,2,2)$. For any set $A$, we define $Q^\alpha_A$ to be the set of tuples $(X,Y,Z)\in \mathcal{P}(A^2)\times \mathcal{P}(A^2)\times \mathcal{P}(A^2)$ such that:
	
	\begin{enumerate}
		\item $X$ and $Y$ are both linear orders with the same $V$ cofinality;
		\item $Z$ is a well order of order type less than $\alpha$; and
		\item The equal cofinality of $X$ and $Y$ is an element of $\Reg_{\text{o.t.}(Z)}$
	\end{enumerate}
\end{definition}

By adding these two $Z$ clauses, we ensure that for $\beta\leq \alpha\leq \On$, we can naturally define $Q^\beta$ from $Q^\alpha$ and $\beta$. We simply have to restrict $Q^\alpha$ to those cases where the third argument of the quantifier defines a set with order type less than $\beta$. This is something which can be defined universally over any $\mathcal{L}$ structure $\mathcal{A}$ which knows even rudimentary set theory and contains all the ordinals below $\beta$. By contrast, if we defined $Q^\alpha$ in the more immediately obvious way (i.e. as a type $(2,2)$ quantifier which was true of any two linear orders $X$ and $Y$ with equal cofinality in $\Reg_{<\alpha}$) then it would not be possible to define $Q^\beta$ from $Q^\alpha$ and $\beta$ in a given structure, unless it believed enough of $\ZFC$ to recreate Definition \ref{definition_Reg_alpha} from scratch with $\Reg$ replaced by the (definable) class $\Reg_{<\alpha}^V$.

Note that $Q^{\On}$ is $Q^{\text{e.c.}}$ with an additional demand for a well-order $Z$ which witnesses the equal cofinality.\footnote{There's one other small difference: $Q^{\On}$ never tells us that two linear orders have the same cofinality if they happen to both have cofinality $\omega$. But in a moderately rich $\mathcal{L}$ structure, that's not a problem: we can discover whether the linear order has cofinality $\omega$ by comparing it to itself using $Q^{\On}$ with different auxiliary sets $Z$. If $Q^{\On}$ is never true, whatever auxiliary set we use, then we know that the order has cofinality $\omega$. And of course any two orders of cofinality $\omega$ have the same cofinality, so we can define the ``missing'' bit of $Q^{\text{e.c.}}$. The details are left to the reader.} We will shortly see that this demand for a witness does not affect the $\LST$ number of the quantifier, at least in the $\On$ case:

\begin{theorem}
	$\LST(I,Q^{\On})=\LST(I,Q^{\text{e.c.}})$
\end{theorem}

Thus, we can reasonably say that (for the purposes of $\LST$ numbers) $Q^{\On}$ is just $Q^{\text{e.c.}}$ expressed in a more convenient form. And therefore, that for $\alpha<\On$, the quantifier $Q^\alpha$ that we have introduced is a fragment of $Q^{\text{e.c.}}$.

Before we get into proving the theorem, we finish off our definitions by introducing the second class of quantifiers we're going to be looking at. These tell us only about cardinalities (not cofinalities), and specifically whether a given cardinality is regular. It uses the same auxiliary set technique as $Q^\alpha$.

\begin{definition}
	Let $\alpha \leq \On$. The quantifier $R^\alpha$ is type $(1,2)$. For any set $A$, we define $Q^\alpha_A$ to be the set of tuples $(X,Z)\in \mathcal{P}(A)\times \mathcal{P}(A^2)$ such that:
	
	\begin{enumerate}
		\item $Z$ is a well order of order type less than $\alpha$; and
		\item $\lvert X \rvert \in\Reg_{\text{o.t.}(Z)}$
	\end{enumerate}
\end{definition}

Morally, $R^\alpha$ is true of $X$ if $\lvert X \rvert \in \Reg_{<\alpha}$; but again we also ask for a well order $Z$ witnessing this. We are interested in looking at the $\LST$ numbers of these predicates together with $I$. (Examining $\LST$ numbers of this kind of quantifier in contexts where $I$ is not available is notoriously difficult: little is known even about $\LST(Q^{\text{e.c.}})$.)

\section{Inequalities}

The theorem we stated above is only one of several inequalities we can prove in $\ZFC$.

\begin{theorem}\label{theorem arranging LST numbers}
	In the following statements, if an $\LST$ number does not exist we consider it to be equal to $\infty$. For $\alpha,\beta \in \On\cup \{\On\}$ and $\beta<\alpha$:
	\begin{enumerate}
		\item\label{LST EC equals On} $\LST(I,Q^{\text{e.c.}})=\LST(I,Q^{\On})$
		\item\label{LST Q>R} $\LST(I,Q^\alpha)\geq \LST(I,R^\alpha)$
		\item\label{LST Qalpha>Qbeta} Either $\LST(I,Q^\alpha)\geq \LST(I,Q^\beta)$, or both $\beta \geq \LST(I,Q^\alpha)$ and $\LST(I,Q^\beta)>\min(\Reg_\beta)$.
		\item\label{LST Ralpha>Rbeta} Either $\LST(I,R^\alpha)\geq \LST(I,R^\beta)$, or both $\beta \geq \LST(I,R^\alpha)$ and $\LST(I,R^\beta)>\min(\Reg_\beta)$.
		\item\label{LST R0} $\LST(I,R^1)=\LST(I,R^0)=\LST(I)$.
	\end{enumerate}
\end{theorem}
We will later see (in Corollary \ref{corollary LST beta alpha inequality works}) that if $\beta<\alpha$, then $\beta <\LST(I,R^\alpha), \LST(I,Q^\alpha)$, unless $\alpha$ is above certain large cardinals. So for $\alpha$ below these large cardinals, the first part of \ref{LST Qalpha>Qbeta} and \ref{LST Ralpha>Rbeta} always holds.

\begin{proof}
	The idea of all these proofs is the same. Although the quantifiers are not definable from one another in an arbitrary language and structure, we show that they \textit{can} be defined from each other in a suitable extension of any language and structure.
	
	Throughout these proofs, we shall use the following formulas (in any language containing the symbols $\in$ and $=$), for use with generalised quantifiers:
	$$\psi_{\in} := v_1\in v_2$$
	$$\psi_{\in}' := v_1\in v_2 \wedge v_1 \in v_3 \wedge v_2\in v_3$$
	
	If $x$ is a variable symbol, we will write $\psi_{\in}(x)$ to denote $\psi_{\in}$ with $v_2$ replaced by $x$, and $\psi_{\in}'(x)$ to denote $\psi_{\in}'$ with $v_3$ replaced by $x$. These formulas define certain useful sets over a given structure, for use with quantifiers. $\psi_\in$ defines the set $v_2$ and $\psi_\in'$ defines the well order $v_3$.
	
	Also, throughout these proofs, let $x,y,z$ denote variable symbols which are not otherwise used in any formulas we are working with.
	
	\begin{proof}[Proof (\ref{LST EC equals On})]
		First, we shall show $\geq$. This is trivial if the left hand side is $\infty$. So suppose that $\LST(I,Q^{\text{e.c.}})=\kappa<\infty$. Let $\mathcal{L}$ be a first order vocabulary of cardinality less than $\kappa$, and let $\mathcal{A}$ be an $\mathcal{L}$ structure of cardinality $\lambda \geq\kappa$. We want to show that $\mathcal{A}$ contains an $\mathcal{L}\cup \{I,Q^{\On}\}$ elementary substructure of size less than $\kappa$. Without loss of generality, assume that $\mathcal{A}$ does not contain any ordinals and is an element of $H_{\lambda^+}$, and that $\mathcal{L}$ does not include the $2$-ary predicate symbol $\in$ or the $1$-ary predicate symbol $P$. Let $\mathcal{L}' = \mathcal{L}\sqcup\{\in, P\}$. Let $\mathcal{A}'$ be the $\mathcal{L}'$ structure with domain $H_{\lambda^+}$, with the interpretation of $\in$ inherited from $V$, the predicate $P$ interpreted as true of elements of $\mathcal{A}$, and all the symbols of $\mathcal{L}$ inherited from $\mathcal{A}$ on the elements of $\mathcal{A}$, and interpreted in some arbitrary way everywhere else.
		
		Let $\mathcal{B}'$ be an $\mathcal{L}'\cup \{I,Q^{\text{e.c.}}\}$ elementary substructure of $\mathcal{A}'$ of cardinality less than $\kappa$ (which exists by assumption), and let $\mathcal{C}'$ be its transitive collapse. Let $j:\mathcal{C}'\rightarrow \mathcal{A}'$ be inverse of the collapsing map, an $\mathcal{L}'\cup \{I,Q^{\text{e.c.}}\}$ elementary embedding. Let
		$$\mathcal{B}:=\{x\in \mathcal{B}': P(x)\}= \mathcal{B}'\cap \mathcal{A}$$
		viewed as an $\mathcal{L}$ structure, and let $\mathcal{C}:=\{x\in \mathcal{C}': P(x)\}$ be the analogous substructure of $\mathcal{C}'$.
		
		
		
		First, we shall prove that $\mathcal{A}'$ and $\mathcal{C}'$ agree on what $\Reg_\alpha$ looks like for all applicable $\alpha$.
		
		\begin{claim}
			
			\begin{enumerate}[label=(\alph*)]
				\item Both $\mathcal{A}'$ and $\mathcal{C}'$ believe a set is an ordinal if and only if it is an ordinal of $V$.\label{Simple LST theorem ord claim}
				\item Both $\mathcal{A}'$ and $\mathcal{C}'$ believe an ordinal is a cardinal if and only if it is a cardinal of $V$.\label{Simple LST theorem card claim}
				\item Both $\mathcal{A}'$ and $\mathcal{C}'$ believe a cardinal is regular if and only if it is a regular cardinal of $V$.\label{Simple LST theorem reg claim}
				\item There is an $\mathcal{L}'\cup \{I,Q^{\text{e.c.}}\}$ formula $\psi_{W}(x,y)$ which both $\mathcal{A}'$ and $\mathcal{C}'$ interpret as true of $(x,y)$ if and only if $y$ is an ordinal and $x\in \Reg_y^V$.\label{Simple LST theorem reg alpha claim}				
				
			\end{enumerate}
		\end{claim}
		\begin{proof}[Proof (Claim, 
			\ref{Simple LST theorem ord claim})] Trivial, since $\mathcal{A}'$ and $\mathcal{C}'$ are transitive.
		\end{proof}
		
		
		
		
		\begin{proof}[Proof (Claim, \ref{Simple LST theorem card claim})] First, notice that since $\mathcal{A}'$ has domain $H_{\lambda^+}$, the cardinals of $\mathcal{A}'$ are precisely the cardinals of $V$ which are in $\mathcal{A}'$.
			
			In both $\mathcal{A}'$ and $\mathcal{C}'$, the $\mathcal{L}'\cup \{I\}$ formula			
			$$\psi_C(x) := x\in \On \wedge \forall y\in x \neg I v_1,v_1 (\psi_{\in}(x),\psi_{\in}(y))$$
			
			\noindent is true of $x$ if and only if $x$ is a cardinal of $V$. So we have a way to ``test'' whether an ordinal is a cardinal. Since $\mathcal{A}$ is correct about cardinals, then,
			$$\mathcal{A}'\vDash \forall x \, x\in \Card \iff \psi_C(x)$$
			
			As usual, ``$x\in \Card$'' is shorthand for the first order formula in the language of set theory which says that $x$ is a cardinal (of $\mathcal{A}$ in this case). Since $j$ is $\mathcal{L}\cup \{I\}$ elementary,
			$$\mathcal{C}'\vDash \forall x \, x\in \Card\iff \psi_C(x)$$
			
			Hence, the cardinals of $\mathcal{C}'$ are precisely the cardinals of $V$ that are in $\mathcal{C}'$.
		\end{proof}
		
		\begin{proof}[Proof (Claim, \ref{Simple LST theorem reg claim})] Similar to the previous case, but instead of $\psi_C$ we use the formula:
			$$\psi_R(x) := \psi_C(x) \wedge \forall y\in x \neg Q^{\text{e.c.}} v_1,v_2,v_1,v_2(\psi_{\in}'(x),\psi_{\in}'(y))$$
			
			\noindent which is true in both $\mathcal{A}'$ and $\mathcal{C}'$ if and only if $x$ is a $V$-regular cardinal. We know that $\mathcal{A}'$ is correct about the regular cardinals because its domain is $H_{\lambda^+}$, so by elementarity $\mathcal{C}'$ is as well.
		\end{proof}
		
		\begin{proof}[Proof (Claim, \ref{Simple LST theorem reg alpha claim})] Since $\Reg^{\mathcal{A}'}$ is a definable subset of $\mathcal{A}'$, and $\mathcal{A}'$ believes enough of $\ZFC$ to do recursion, we know that $\mathcal{A}'$ can recursively calculate the Cantor-Bendixson ranks of elements of $\Reg^{\mathcal{A}'}$ and therefore can calculate $\Reg_{\alpha}^{\mathcal{A}'}$ for all $\alpha\in \mathcal{A}'$. Let $\psi_W(x,y)$ be the (first order) formula which does this. That is, $\mathcal{A}'$ believes that $\psi_W(x,y)$ is true if and only if $y\in \On$ and $x\in \Reg_y^{\mathcal{A}'}$. Moreover, we have also seen that $\Reg^{\mathcal{A}'}=\Reg^V\cap \On^{\mathcal{A'}}$. Since $\mathcal{A}'$ is transitive, it is easy to verify that for all $y\in \On^{\mathcal{A}'}$, $$\Reg_y^{\mathcal{A}'}=\Reg_y^V\cap \On^{\mathcal{A}'}.$$
			
		So $\psi_W$ has the properties we asked for regarding $\mathcal{A}'$. Since $j$ is elementary, $\psi_W$ also calculates $\Reg_\alpha^{\mathcal{C}'}$ when interpreted over $\mathcal{C}'$. Since $\Reg^{\mathcal{C}'}=\Reg^V\cap \On^{\mathcal{C}'}$ and $\mathcal{C}'$ is transitive, the same argument shows that $\psi_W$ has the desired properties for $\mathcal{C}'$ as well.
		\end{proof}
		
		This ends the proof of the claim. Now we show that if $\varphi$ is an $\mathcal{L}\cup \{I,Q^{\On}\}$ formula then we can find an $\mathcal{L}'\cup \{I,Q^{\text{e.c.}}\}$ formula $\psi$ such that for any assignment $a$ of variables to elements of $\mathcal{A} \subset \mathcal{A}'$,
		\begin{equation}\tag{$\star$}
			\mathcal{A}\vDash \varphi[a]\iff \mathcal{A}' \vDash \psi[a]		\end{equation}
		
		and likewise for $\mathcal{C}'$. We use induction on the length of $\varphi$. The only difficult case is the quantifiers. Even then, $\forall$, $\exists$ and $I$ are easy, since they are in the target language, and we can use $P$ to restrict the variables we quantify over to elements of $\mathcal{A}$ and $\mathcal{C}$. So the only case to deal with is when $\varphi$ has the form
		
		$$\varphi=Q^{\On}v_{i_1},\ldots,v_{i_6} (\varphi_1,\varphi_2,\varphi_3)$$
		
		If $i_1=i_2$, $i_3=i_4$ or $i_5=i_6$ then $\varphi$ is trivial. And it makes no difference whether any other variables used are the same. So without loss of generality, we can simplify notation by assuming that $i_1=1,i_2=2,i_3=1,i_4=2,i_5=1,i_6=2$. 
		
		Let $\psi_k$ be the $\mathcal{L}\cup \{I,Q^{\text{e.c.}}\}$ formula corresponding to $\varphi_k$, which we already have by the inductive hypothesis. Without loss of generality, we can assume that $v_3$ is not a free variable of $\varphi_1$, $\varphi_2$, $\varphi_3$, $\psi_1$, $\psi_2$ and $\psi_3$. To simplify notation, we will write $\psi_k(y,z)$ to denote $\psi_k$ with $y$ and $z$ substituted for $v_1$ and $v_2$ respectively. Let

		$$\chi(v_3):=v_3\in \On \wedge \exists x \big(\forall y \forall z \psi_3(y,z)\iff (y,z)\in x\big) \wedge x \text{ is a well order } \wedge \text{o.t.}(x)=v_3$$
		
		
		Unpicking this, $\psi_1$ says that $\psi_3$ defines a well order, and that well order is an element of the structure, and it has order type $v_3$.
		
		\begin{claim}
			Let $a$ be an assignment of variables, other than $x$ and $y$, to elements of $\mathcal{A}$ (resp. $\mathcal{C}$). If $\psi_3(y,z)[a]$ defines a linear order over $\mathcal{A}$ (resp. $\mathcal{C}'$) then that linear order will be an element of $\mathcal{A}'$ (resp. $\mathcal{C}'$). If the linear order is a well order, then the ordinal corresponding to its order type will also be in $\mathcal{A}'$ (resp. $\mathcal{C}'$). 
		\end{claim}
		\begin{proof}
			For $\mathcal{A}$, this is immediate, because the set defined by $\psi_3$ is a subset of $\mathcal{A}^2$, and we know the domain of $\mathcal{A}'$ is $H_{\lambda^+}$ and the domain of $\mathcal{A}$ has size $\lambda$.
			
			Suppose that $\psi_3(y,z)[a]$ defines a linear order over $\mathcal{C}'$. Then by $\mathcal{L}'\cup \{I,Q^{\text{e.c.}}\}$ elementarity of $j$, $\psi_3(y,z)[j(a)]$ defines a linear order over $\mathcal{A}'$. Also, since $a$ assigns variables to elements of $\mathcal{C}$, we know $j(a)$ assigns variables to elements of $\mathcal{A}$. So by the previous paragraph, the linear order defined by $\psi_3(y,z)[j(a)]$ is in $\mathcal{A}'$. But then by elementarity of $j$, the linear order defined by $\psi_3(y,z)[a]$ is in $\mathcal{C}'$. The well order statement is similar.
		\end{proof}

		Hence, for a given assignment $a$ (which omits $v_4,v_5$ to simplify notation) both $\mathcal{A}'$ and $\mathcal{C}'$ believe $\chi(v_3)[a]$ if and only if the set defined by $\varphi_3[a]$ is a well order of order type $v_3$. We now define
			\begin{align*}
				\psi :=	
				& \exists v_3 \in \On \chi(v_3) \wedge  Q^{\text{e.c.}}v_1,v_2,v_1,v_2(\psi_1,\psi_2)\wedge\\
				&\wedge \exists x \in \On \psi_W(v_3,x) \wedge Q^{\text{e.c.}}v_1,v_2,v_1,v_2(\psi_1, \psi_{\in}'(x))
			\end{align*}
		
		Verifying that $\psi$ has the desired property is now an easy definition chase. Fix an assignment $a$ of variables (other than $v_1$ and $v_2$) to elements of $\mathcal{A}$. Let $X,Y,Z\subset \mathcal{A}$ be the objects defined by $\varphi_1[a]$, $\varphi_2[A]$ and $\varphi_3[a]$ over $\mathcal{A}$, which by inductive hypothesis are also defined by $\psi_1[a]$, $\psi_2[a]$ and $\psi_3[a]$ over $\mathcal{A}'$.
		\begin{align*}
			\mathcal{A}\vDash\varphi[a] \iff V\vDash& Z \text{ is a well order and } X \text{ and } Y \text{ have equal cofinality in } \Reg_{\text{o.t.}(Z)}\\
			\iff V\vDash&\exists \beta \in \On^{\mathcal{A}'} \text{o.t.}(Z)=\beta \wedge \cof(X)=\cof(Y) \wedge\\
			&\wedge \exists \gamma \in \Reg_\beta^V \cof(X)=\gamma=\cof(\gamma)\\
			\iff \mathcal{A}'\vDash& \psi
		\end{align*}
		
		The $\mathcal{C}$ case is identical.
		
		But we know that $j$ is elementary, so for any assignment $a$ of variables to $\mathcal{C}$,
		
		\begin{align*}
			\mathcal{C}\vDash\varphi[a] &\iff \mathcal{C}' \vDash \psi[a]\\
			&\iff \mathcal{A}' \vDash \psi[j(a)]\\
			&\iff \mathcal{A} \vDash \varphi[j(a)]
		\end{align*}
		
		Hence, $j\upharpoonleft \mathcal{C}$ is an $\mathcal{L}\cup \{I,Q^{\On}\}$ elementary embedding from $\mathcal{C}$ to $\mathcal{A}$. Hence, $\kappa \geq \LST(I,Q^{\On})$.

		Now we show $\LST(I,Q^{\text{e.c.}})\leq \LST(I,Q^{\On})$. This uses a similar trick, but the technique is much simpler. Let $\mathcal{L}$ be a first order language of size less than $\kappa$, and let $\mathcal{A}$ be an $\mathcal{L}$ structure of cardinality $\lambda>\kappa=\LST(I,Q^{\On})$. Without loss of generality, assume $\mathcal{A}$ contains no ordinals, and $<$ is not an element of $\mathcal{L}$. Let $\mathcal{A}'$ be the $\mathcal{L}'=\mathcal{L}\cup\{\in\}$ structure $\mathcal{A}\sqcup (\lambda+1,\in)$, where $\in$ is interpreted as true only on pairs of ordinals and not on elements of $\mathcal{A}$. (Note that $\mathcal{L}'$ can distinguish between the elements of $\mathcal{A}$ and the ordinals in $\mathcal{A}'$ by using $\in$, without the need for the extra predicate $P$ we used before.) Let $\mathcal{B}'$ be an $\mathcal{L}'\cup\{I,Q^{\On}\}$ elementary substructure of $\mathcal{A}'$ of size less than $\kappa$, and let $\mathcal{B}=\mathcal{B}'\cap \mathcal{A}$.
		
		Now, $\mathcal{A}'$ believes: ``The cardinality of $\mathcal{A}$ is equal to the cardinality of the set of ordinals below my largest ordinal.'' This statement can easily be expressed in $\mathcal{L}'\cup \{I\}$, and hence $\mathcal{B}'$ believes the corresponding statement about $\mathcal{B}$. Hence, if $\beta \leq \lvert \mathcal{B}\rvert$, then we can find some ordinal $\gamma \in\mathcal{B}'$ such that $\gamma \cap \mathcal{B}'$ has order type $\beta$. The same is also clearly true of $\mathcal{A}$ and $\mathcal{A}'$.
		
		We now prove the existence of a formula $\psi$ satisfying $\star$ for $\mathcal{A}$ and $\mathcal{B}$. All the cases except $Q^{\text{e.c.}}$ are trivial. Suppose 
		$$\varphi=Q^{\text{e.c.}}v_1,v_2,v_1,v_2(\varphi_1,\varphi_2)$$ and $\psi_1, \psi_2, \psi_3$ correspond to $\varphi_1,\varphi_2$ as described in $\star$. Then we define
		$$\psi:=\exists x\in \On Q^{\On}v_1,v_2,v_1,v_2,v_1,v_2(\psi_1,\psi_2,\psi_{\in}'(x))$$
		$v_3\in \On$ is shorthand for ``$\exists y (y \in x \vee x\in y)$''. This definition of $\psi$ satisfies $\star$ because any linear order over a structure of size $\mu$ will have cofinality in $\Reg_{\leq \mu}$.
		
		The proof now goes through in exactly the same way as $\geq$, and we conclude that $\LST(I,Q^{\text{e.c.}})=\LST(I,Q^{\On})$.
	\end{proof}
	
	\begin{proof}[Proof (\ref{LST Q>R}): $\LST(I,Q^\alpha) \geq \LST(I,R^\alpha)$] We copy the construction from the $\leq$ part of the previous proof. Let $\mathcal{A}$ be an $\mathcal{L}$ structure with cardinality $\lambda\geq \kappa=\LST(I,Q^\alpha)$; let $\mathcal{A}'=\mathcal{A}\sqcup (\lambda,\in)$ viewed as an $\mathcal{L}'=\mathcal{L}\sqcup\{\in\}$ structure with $\in$ only true about pairs of ordinals, let $\mathcal{B}' \prec \mathcal{A}'$ have size less than $\kappa$, and let $\mathcal{B}=\mathcal{B}'\cap \mathcal{A}$. As before, we know that both $\mathcal{A}'$ and $\mathcal{B}'$ contain as many ordinals below their top as there are elements of $\mathcal{A}$ and $\mathcal{B}$ respectively.
			
		Again, we show that for any $\mathcal{L}\cup \{I,R^\alpha\}$ formula $\varphi$, we can find a formula $\psi$ such that $\star$ holds. The only interesting case is
		$$\varphi=R^\alpha v_1,v_1,v_2(\varphi_1,\varphi_2)$$
		Assuming that we have found $\psi_1$ and $\psi_2$ corresponding to $\varphi_1$ and $\varphi_2$ according to $\star$, we let
		\begin{align*}
			\psi:=\exists x \in \On &I v_1,v_1(\psi_1,\psi_{\in}(x)) \wedge\\
			& \forall y \in x \neg I v_1,v_1(\psi_{\in}(x),\psi_{\in}(y))\\
			&Q^\alpha v_1,v_2,v_1,v_2,v_1,v_2(\psi_{\in}'(x),\psi_{\in}'(x),\psi_2)\\
			&\forall y\in x \neg Q^\alpha v_1,v_2,v_1,v_2,v_1,v_2(\psi_{\in}'(x),\psi_{\in}'(y),\psi_2)
		\end{align*} 
		
		It is routine to verify that $\psi$ satisfies $\star$. As before, this lets us conclude that $\LST(I,R^\alpha)\leq \LST(I,Q^\alpha)$.
	\end{proof}
	
	\begin{proof}[Proof (\ref{LST Qalpha>Qbeta})] First suppose that $\beta<\LST(I,Q^\alpha)$. We shall show that $\LST(I,Q^\alpha)\geq \LST(I,Q^\beta)$. The technique is (once again) similar to the previous ones. However, we need to add enough constants to preserve $\beta$ when we take an elementary substructure.
		
	Let $\mathcal{L}$ be a first order language of size less than $\LST(I,Q^\alpha)$, and let $\mathcal{A}$ be an $\mathcal{L}$ structure of size $\lambda\geq\LST(I,Q^\alpha)$ containing no ordinals. Expand $\mathcal{L}$ to a language $\mathcal{L}'=\mathcal{L}\sqcup\{\in\}\sqcup\{c_i: i<\beta+1\}$. Note that $\mathcal{L}'$ still has cardinality less than $\LST(I,Q^\alpha)$. Let $\mathcal{A}'=\mathcal{A}\sqcup (\lambda+1,\in,0,1,\ldots,\beta)$ be an expansion of $\mathcal{A}$ to an $\mathcal{L}'$ structure which adds the first $\lambda+1$ ordinals, assigns the constant symbol $c_i$ to $i\in \mathcal{A}'$, and interprets $\in$ as true only on ordinals. Let $\mathcal{B}'\leq \mathcal{A}'$ be an $\mathcal{L}'\cup \{I,Q^\alpha\}$ elementary substructure of cardinality $<\LST(I,Q^\alpha)$, and let $\mathcal{B}=\mathcal{B}'\cap \mathcal{A}$.
	
	As usual, we prove by induction that for any $\mathcal{L}\cup \{I,Q^\beta\}$ formula $\varphi$ we can find an $\mathcal{L}\cup \{I,Q^\alpha\}$ formula $\psi$ satisfying $\star$ for $\mathcal{A}'$ and $\mathcal{B}'$. The only interesting case is when
	$$\varphi=Q^\beta v_1,v_2,v_1,v_2,v_1,v_2(\varphi_1,\varphi_2,\varphi_3)$$
	If $\psi_1,\psi_2,\psi_3$ are the formulas corresponding to $\varphi_1,\varphi_2,\varphi_3$, then we take
	$$\psi=Q^\alpha v_1, v_2, v_1,v_2, v_1 v_2(\psi_1,\psi_2,\psi_3) \wedge \exists x \in c_\beta Q^\alpha v_1, v_2, v_1, v_2, v_1, v_2(\psi_1,\psi_2, \psi_{\in}(x)$$
	
	The rest of the proof is standard.
		
	Now suppose instead that $\LST(I,Q^\beta)\leq\min(\Reg_\beta)$, and suppose (seeking a contradiction) that $\LST(I,Q^\alpha)<\LST(I,Q^\beta)$. Let $\mathcal{A}$ be an $\mathcal{L}$ structure, for some first order language $\mathcal{L}$ of cardinality less than $\LST(I,Q^\alpha)$). Then $\mathcal{A}$ contains an $\mathcal{L}\cup\{I,Q^\beta\}$ elementary substructure $\mathcal{B}$ of cardinality less than $\LST(I,Q^\beta)$, and $\mathcal{B}$ has an $\mathcal{L}\cup \{I,Q^\alpha\}$ substructure $\mathcal{C}$ of cardinality less than $\LST(I,Q^\alpha)$. But $\lvert \mathcal{B}\rvert<\min (\Reg_\beta)$, so every linear order we can construct from its elements has cofinality in $\Reg_{<\beta}$. Hence, in $\mathcal{B}$ and $\mathcal{C}$, $Q^\alpha$ and $Q^\beta$ are interpreted in exactly the same way. So $\mathcal{C}$ is an $\mathcal{L}\cup \{I,Q^\beta\}$ elementary substructure of $\mathcal{B}$ and hence of $\mathcal{A}$. So any $\mathcal{L}$ structure contains an $\mathcal{L}\cup \{I,Q^\beta\}$ elementary substructure of cardinality $<\LST(I,Q^\alpha)<\LST(I,Q^\beta)$. Contradiction.
	\end{proof}
	
	\begin{proof}[Proof (\ref{LST Ralpha>Rbeta})] Just like the previous case.
	\end{proof}
	
	\begin{proof}[Proof (\ref{LST R0})] First note that $R^0$ is simply always false, since there are no ordinals below $0$. $\LST(I,R^1)\geq \LST(I)$ is also trivial. The only thing to show is $\LST(I,R^1)\leq\LST(I)$. As usual, let $\mathcal{L}$ be a language of size $<\LST(I)$, let $\mathcal{A}$ be an $\mathcal{L}$ structure of size $\lambda \geq \LST(I)$ with no ordinals, let $\mathcal{L}' = \mathcal{L}\sqcup \{\in\}$, let $\mathcal{A}'=\mathcal{A}\sqcup (\lambda+1,\in)$, let $\mathcal{B}'$ be an $\mathcal{L}'\cup \{I\}$ elementary substructure of $\mathcal{A}'$, and let $\mathcal{B}=\mathcal{B}'\cap \mathcal{A}$. Once more, we prove that for any $\mathcal{L}\cup \{I,R^1\}$ formula $\varphi$, we can find an $\mathcal{L}'\cup \{I\}$ formula $\psi$ satisfying $\star$ for $\mathcal{A}$ and $\mathcal{B}$. The only interesting case is
	$$\varphi=R^1 v_1, v_1, v_2(\varphi_1,\varphi_2)$$
	which is true if and only if $\varphi_1$ defines a set of successor cardinality, and $\varphi_2$ defines the empty set. So we use
	\begin{align*}
		\psi:=&(\forall v_1\forall v_2 \neg \psi_2) \wedge  \exists x\in \On \big(\\
		&I v_1,v_1(\psi_1,\psi_{\in}(x)) \wedge \\
		&\wedge \forall y\in x \neg I v_1,v_1(\psi_{\in}(x),\psi_{\in}(y)) \wedge\\
		&\wedge \exists z\in x \forall y\in x z\in y \implies I v_1, v_1 (\varphi_{\in}(z),\varphi_{\in}(y))
	\end{align*}
	The rest of the proof uses the same methods as the other parts of the` theorem.
	\end{proof}
\end{proof}

This theorem shows that if we exclude very large values of $\alpha$, $\LST(I,Q^\alpha)$ and $\LST(I,R^\alpha)$ move downward as we decrease $\alpha$ or move from $Q$ to $R$. At the bottom of this hierarchy is $\LST(I)$, and we saw earlier (from \cite{magidorVaananan}) that the minimum possible value of this is precisely the first inaccessible. At the top is $\LST(I,Q^{\On})=\LST(I,Q^{\text{e.c.}})$, and we saw that the minimum possible value of this is precisely the first Mahlo cardinal. Our goal in this paper is to find a similar result for the other members of the hierarchy. The tagline is:

\begin{quote}
	For $\alpha>0$, the minimum consistent values for $\LST(I,Q^\alpha)$ and $\LST(I,R^\alpha)$ are both precisely the first element of $\Reg_{\alpha}$, provided $\alpha$ is not too large.
\end{quote}

So for example, this says the minimum value of $\LST(I,R^1)=\LST(I)$ is the first element of $\Reg_1$, which is the least inaccessible as we had expected. We will, of course, explain what we mean by ``too large'' shortly; it depends on whether we're dealing with $Q$ or $R$.

First, we shall prove that these cardinals are lower bounds for the $\LST$ numbers.

\begin{theorem}\label{theorem LST R minima}
	Let $\alpha >0$ be such that there are no hyperinaccessibles below $\alpha$. Then if it exists, $\LST(I,R^\alpha)$ is at least the first element of $\Reg_\alpha$.
	
\end{theorem}
\begin{proof}
	Suppose $\kappa=\LST(I,R^\alpha)$ is less than the first element of $\Reg_{\alpha}$. Let $\mathcal{A}=( H_{\kappa^+},\in)$. The cardinals of $\mathcal{A}$ are precisely the cardinals of $V$ up to (and including) $\kappa$.
	
	By definition of the $\LST$ number, we can find an elementary substructure $\mathcal{B}\subset \mathcal{A}$ of cardinality less than $\kappa$. Since $\mathcal{A}$ is well founded, so is $\mathcal{B}$, so we can take its transitive collapse $\mathcal{C}=(C,\in)$ and the elementary embedding $j: \mathcal{C}\rightarrow \mathcal{A}$.
	
	Now, letting $\psi_{\in}$ be as in the previous theorem, if $a$ is an assignment sending $v_3$ to $\epsilon$,
	\begin{align*}
		\epsilon \in \Card^{\mathcal{C}}&\iff j(\epsilon) \in \Card^{\mathcal{A}}\\
		&\iff j(\epsilon) \in \Card^V\\
		&\iff\mathcal{A}\vDash \forall x\in v_3\, \neg  I v_1,v_1(\psi_{\in}(v_3),\psi_{\in}(x))[j(a)]\\
		&\iff \mathcal{C}\vDash\forall x\in v_3\, \neg I v_1, v_1(\varphi_{\in}(v_3),\psi_{\in}(x))[a]\\
		&\iff \epsilon \in \Card^V
	\end{align*}
	The last $\iff$ follows because $\mathcal{C}$ is transitive. Similarly, 
	\begin{align*}
		\epsilon \in \Reg^{\mathcal{C}}&\iff j(\epsilon) \in \Reg^\mathcal{A}\\
		&\iff j(\epsilon) \in (\Reg_{<\alpha})^V\\ 
		&\iff \mathcal{A} \vDash \exists x\in \On R^\alpha v_1, v_1, v_2(\psi_{\in}(v_3), \psi_{\in}'(x))[j(a)]\\
		&\iff \mathcal{C}\vDash \exists x\in \On R^\alpha v_1, v_1, v_2 (\psi_{\in}(v_3), \psi_{\in}'(x))[a]\\
		&\iff \epsilon \in (\Reg_{<\alpha})^V\\
		&\iff \epsilon \in \Reg^V
	\end{align*}
	
	$\mathcal{A}$ and $\mathcal{C}$ can then (using the same recursion) calculate which cardinals are in $\Reg_{\beta}^V$ for each $\beta<\alpha$.
	
	Now, let $\epsilon\in \mathcal{C}$ be the least ordinal such that $j(\epsilon)>\epsilon$. (Such a $\epsilon$ must exist: $\mathcal{C}$ has a largest cardinal $\delta$ which must be smaller than $\kappa$, and $j(\delta)=\kappa>\delta$.)
	
	Clearly $\epsilon$ is a cardinal of $\mathcal{C}$. Otherwise there would be some bijection $f:\gamma\rightarrow \epsilon$ for some $\gamma<\epsilon$; and then $j(f)=f$ would be a bijection from $j(\gamma)=\gamma$ to $j(\epsilon)>\epsilon$, which is clearly nonsense. So $\epsilon$ is a cardinal of $\mathcal{C}$, and hence of $V$. Moreover, $\epsilon$ cannot be a successor cardinal: if $\epsilon=(\delta^+)^{\mathcal{C}}$ then
	$$j(\epsilon)=(j(\delta)^+)^{\mathcal{A}}=(j(\delta)^+)^V=(\delta^+)^V=(\delta^+)^{\mathcal{C}}=\epsilon$$
	
	Suppose that $\epsilon$ is singular (in $\mathcal{C}$ or equivalently in $V$). Then we can take some sequence $\mu=(\mu_\beta)_{\beta<\delta}\in \mathcal{C}$ which is cofinal in $\epsilon$, with $\delta<\epsilon$. Since $j{\upharpoonleft} \epsilon=\id$, $j(\mu_\beta)=\mu_\beta$ for all $\beta$, and the length of $j(\mu)$ is $j(\delta)=\delta$. Hence $j(\mu)=\mu$. But $j(\mu)$ is cofinal in $j(\epsilon)$, so then $j(\epsilon)=\epsilon$. Contradiction.
	
	So $\epsilon$ is a regular limit cardinal, and hence is in $\Reg_\gamma^V$ for some $0<\gamma$. We know that $\gamma<\alpha$ since $\lvert \mathcal{C}\rvert < \min (\Reg^V_\alpha)$. As usual, since $\epsilon \in \Reg_\gamma$ we know $\gamma\leq\epsilon$. Moreover, if $\gamma=\epsilon$ then it is a hyperinaccessible, and we are assuming none exist below $\alpha$. So $\gamma<\epsilon$ and $j(\gamma)=\gamma$.
	
	Say $\epsilon$ is the $\delta$'th element of $\Reg_\gamma$. Since $\Reg_\gamma\cap \epsilon$ must be bounded below $\epsilon$, we know $\delta<\epsilon$ and hence $j(\delta)=\delta$. But then
	$$\mathcal{C}\vDash ``\epsilon \text{ is the }\delta\text{'th element of}\Reg_\gamma^{\mathcal{C}}"$$
	
	and hence by elementarity
	$$\mathcal{A}\vDash ``j(\epsilon) \text{ is the }\delta\text{'th element of}\Reg_\gamma^{\mathcal{A}}"$$
	
	But this is saying that $\epsilon$ and $j(\epsilon)$ are both the same element of $\Reg_\gamma^V$. Contradiction.
\end{proof}

By the above and Theorem \ref{theorem arranging LST numbers} it also follows that $\LST(I,Q^\alpha)\geq \min (\Reg_\alpha)$ under the above conditions. In fact, if $\min (\Reg_\alpha)$ is strongly inaccessible (e.g. because we are assuming GCH) then this will hold under weaker conditions:

\begin{theorem}
	Let $\alpha>0$ and suppose that the first element of $\Reg_\alpha$ is strongly inaccessible. Suppose there are no Mahlo cardinals below $\alpha$. Then if it exists, $\LST(I,Q^\alpha)$ is at least the first element of $\Reg_{\alpha}$.
\end{theorem}
\begin{proof}
	Suppose that $\LST(I,Q^\alpha)<\min(\Reg_{\alpha})=:\kappa$. Since $\kappa$ is strongly inaccessible, $M:=H_{\kappa}$ is a model of $\ZFC$. It is easy to see that $\LST^M(I,Q^\alpha)\leq \LST^V(I,Q^\alpha)\in M$: if $\mathcal{A}\in M$ is an $\mathcal{L}$ structure, then it contains some $\mathcal{L}\cup \{I,Q^\alpha\}$ elementary substructure $\mathcal{B}\in V$ of cardinality $<\LST(I,Q^\alpha)^V$. Then $\mathcal{B}\in H_\kappa=M$.
	
	Moreover, $M$ contains no Mahlo cardinals. This is because (by assumption) it has none below $\alpha$, and since any Mahlo cardinal is hyperinaccessible, there cannot be any in the interval $[\alpha,\min (\Reg_{\alpha}))=[\alpha,\kappa)$. So we know that $\LST^M(I,Q^{\text{e.c.}})=\infty$.
	
	But $M$ doesn't contain any cardinals which are in $\Reg_{\alpha},\Reg_{\alpha+1},\ldots$. So as far as $M$ is concerned $Q^{\alpha}$ is evaluated in exactly the same way as $Q^{\On}$. Hence, $\LST^M(I, Q^\alpha)= \LST^M(I,Q^{\On})=\infty$. Contradiction.
\end{proof}

A similar argument shows that if there are no Mahlo cardinals below $\alpha$ and $\Reg_\alpha = \emptyset$ then $\LST(I,Q^\alpha)=\infty$.

This gives us the promised largeness criterion, necessary for the bad option of \ref{LST Qalpha>Qbeta} and \ref{LST Ralpha>Rbeta} of Theorem \ref{theorem arranging LST numbers} to hold:

\begin{corollary}\label{corollary LST beta alpha inequality works}
	If $\beta<\alpha$ and there are no hyperinaccessibles below $\alpha$ then $\LST(I,R^\beta)\leq \LST(I,R^\alpha)$ and $\LST(I,Q^\beta)\leq \LST(I,Q^\alpha)$. If instead there are no Mahlo cardinals below $\alpha$ and $\min (\Reg_\alpha)$ is strongly inaccessible, then $\LST(I,Q^\beta)\leq \LST(I,Q^\alpha)$.
\end{corollary}
\begin{proof}
	If $\beta=0$ then this is trivial, as both $Q^0$ and $R^0$ are always false. So assume $\beta>0$.
	
	We address the $Q$ case first. We are assuming there are no Mahlo cardinals below $\alpha$, and we have just seen that this implies $\LST(I,Q^\alpha)\geq \min (\Reg_\alpha)$. But $\min (\Reg_\alpha) \geq \alpha>\beta$, so $\beta<\LST(I,Q^\alpha)$. We saw in Theorem \ref{theorem arranging LST numbers} that this implies $\LST(I,Q^\beta)\leq \LST(I,Q^\alpha)$.
	
	The $R$ case follows in exactly the same way.
\end{proof}

So we have seen that if $\alpha$ is not too large, the first element of $\Reg_{\alpha}$ is a lower bound for $\LST(I,Q^\alpha)$ and $\LST(I,R^\alpha)$. We now want to show these lower bounds are optimal. This our main result, and the rest of this paper is devoted to its proof.

Just before we begin, we shall note in passing that there is also an upper bound for $\LST(I,Q^\alpha)$ and $\LST(I,R^\alpha)$: they can never be above a supercompact. The proof is fairly simple, but introduces several of the techniques we will use in proving the main theorem.

\begin{theorem}\label{theorem LST maximum}
	Let $\kappa$ be a supercompact, and let $\alpha<\kappa$. Then $\LST(I,Q^\alpha)$ exists, and is no larger than $\kappa$. Hence, the same is true of $\LST(I,R^\alpha)$.
\end{theorem}
\begin{proof}
	Let $\mathcal{L}$ be a first order language of cardinality less than $\kappa$, and let $\mathcal{A}\in V$ be an $\mathcal{L}$ structure, of cardinality $\lambda\geq \kappa$. Without loss of generality, let us say that the domain of $\mathcal{A}$ is $\lambda$. Without loss of generality, let us further say that $\mathcal{L}\in H_\kappa$. We want to show that there is a substructure of cardinality less than $\kappa$.
	
	Let $j:V\rightarrow M$ be an elementary embedding with critical point $\kappa$, such that $j(\kappa)>\lambda$ and $M^\lambda\subset M$. (Such an embedding exists by definition of a supercompact). Note that since $\mathcal{L}\in H_\kappa$, $j$ acts as the identity on $\mathcal{L}$. Moreover, $\mathcal{A}$ can be coded as a $\lambda$ sequence. Since $M$ contains all its $\lambda$ sequences, we know that $\mathcal{A}\in M$. Also notice that since $\alpha<\kappa\leq\lambda$ we know that $j((Q^\alpha)^V)=(Q^\alpha)^M$. And of course, $j(I^V)=I^M$.
	
	Another consequence of $M$ containing all its $\lambda$ sequences is that it correctly calculates the cardinalities and cofinalities of all subsets of $\lambda$ and $\lambda \times \lambda$. In particular, both $I$ and $Q^\alpha$ are evaluated the same way over $\mathcal{A}$ in both $M$ and $V$. It follows that if $\Phi$ is an $\mathcal{L}\cup \{I,Q^\alpha\}$ formula (with parameters from $\mathcal{A}$ and $\mathcal{L}$) then
	$$(\mathcal{A} \vDash \Phi)^M \iff (\mathcal{A}\vDash \Phi)^V$$
	
	By elementarity, we know that
	$$(\mathcal{A}\vDash \Phi)^V \iff (j(\mathcal{A}) \vDash \Phi)^M$$
	
	There is a small technicity which is disguised by this notation. $\Phi$ can have parameters from $\mathcal{A}$ and $\mathcal{L}$, which we have suppressed. For elementarity to hold, the $\Phi$ on the right must take the images of those parameters under $j$. But in fact, since $j$ acts as the identity on $\mathcal{L}$ and on the domain $\lambda$ of $\mathcal{A}$, doing $j$ to those parameters just gives us back the original parameters.
	
	This argument implies that in $M$,
	$$\mathcal{A}\vDash \Phi \iff j(\mathcal{A})\vDash \Phi$$
	
	Hence, from the perspective of $M$, $\mathcal{A}$ is an $\mathcal{L}\cup \{I,Q^\alpha\}$ elementary substructure of $j(\mathcal{A})$ of cardinality $\lambda<j(\kappa)$. So $M$ believes: ``$j(\mathcal{A})$ contains an $\mathcal{L}\cup \{I,Q^\alpha\}$ elementary substructure of cardinality less than $j(\kappa)$.'' So by elementarity, $V$ believes: ``$\mathcal{A}$ contains an $\mathcal{L}\cup \{I,Q^\alpha\}$ elementary substructure of cardinality less than $\kappa$.'' Which is what we wanted to show.
\end{proof}

\section{The main theorem}

\begin{theorem}\label{theorem LST main result}
	Let $0<\alpha\in \On$. Suppose it is consistent that there is a supercompact cardinal larger than $\alpha$ and GCH holds. Then it is also consistent that $\alpha$ is countable (and so has no hyperinaccessibles below it) and $\LST(I,Q^\alpha)$ and $\LST(I,R^\alpha)$ are both no larger than the first element of $\Reg_{\alpha}$.
\end{theorem}

\begin{proof}
	
We adapt a technique from \cite{magidorVaananan}. The first step is to show that without loss of generality, we may assume that $\epsilon$ is countable. Take a universe $V_0$ which believes GCH and has a supercompact $\kappa>\epsilon$. Let $V$ be a generic extension of $V_0$ by the collapsing forcing $\col(\omega,\lvert\epsilon\rvert)$. We show that $V$ inherits the properties of $V_0$: $\kappa$ is supercompact, and GCH holds. Preservation of GCH is standard, but the supercompactness of $\kappa$ requires a little work. Recall the \textit{Silver Lifting Criterion}:

\begin{lemma}\label{lemma j and j*}\cite{cummings_handbook}[9.1](Silver)
	Let $j:\tilde{V}\rightarrow M$ be an elementary embedding from some universe $\tilde{V}$ to a model $M$. Let $\PP\in \tilde{V}$ be a forcing. Let $G$ be $\PP$ generic, and let $H$ be $j(\PP)$ generic. Suppose that for all $p\in G$, $j(p)\in H$. Then there is an elementary embedding $j^*: \tilde{V}[G]\rightarrow M[H]$, which extends $j$.
	
			
\end{lemma}

If we add some extra assumptions, we can also prove that $j^*$ preserves $\lambda$ sequences like a supercompact embedding.

\begin{lemma}
	Suppose the conditions of the above lemma hold, that $j(\PP)=\PP$, $H=G$ and that $\PP$ satisfies the $\lambda^+$ chain condition. Then $\tilde{V}[G]$ believes that $M[G]^\lambda \subset M[G]$.
\end{lemma}

In fact, this can be proved in much more general circumstances: rather than assuming that $j(\PP)=\PP$ and $H=G$ we only need that $G\in M[H]$ and $H\in \tilde{V}[G]$. But that's overkill for this proof.

\begin{proof}
	First, note that since $M$ is definable in $\tilde{V}$, $M[G]$ is definable in $V[G]$. Let $S=(S_\gamma)_{\gamma<\lambda}\in \tilde{V}[G]$ be a $\lambda$ sequence of elements of $M[G]$. Since $M$ and $G$ are definable in $\tilde{V}[G]$, we can also define a sequence $\dot{S}=(\dot{S}_\gamma)_{\gamma<\lambda}\in \tilde{V}[G]$ of names for the terms of $S$. (Note that $S$ itself is not a name, it's just a sequence of names.) Let $\dot{\Sigma}\in \tilde{V}$ be a $\PP$ name for $\dot{S}$. Let $p\in G$ force $\dot{\Sigma}$ to be a $\lambda$ sequence of $\PP$ names, each of which is an element of $M$.
	
	For $\gamma<\lambda$, let $A_\gamma$ be a maximal antichain of conditions below $p$ which decide which name $\dot{\Sigma}(\gamma)$ is going to be interpreted as. So for any $q\in A_\gamma$ there is some name $\sigma_{q,\gamma}\in M$ such that $q\forces \dot{\Sigma}(\gamma)=\sigma_{q,\gamma}$. Then $\dot{\Sigma}$ is forced by $p$ to be equal to
	$$\tau := \{\langle (\sigma_{q,\gamma},\check{\gamma}), q\rangle : \gamma<\lambda, q\in A_\gamma\}$$
	
	By the chain condition, $\lvert A_\gamma\rvert \leq \lambda$ for all $\gamma<\lambda$, and hence $\tau$ has cardinality $\leq \lambda$. All the elements of $\tau$ are in $M$, $\tau\in \tilde{V}$ and $M$ is closed under $\lambda$ sequences, so $\tau \in M$. So $\dot{S}=\tau^G\in M[G]$. The conclusion that $S\in M[G]$ is now immediate.
\end{proof}

	\begin{corollary}
		$\kappa$ is supercompact in $V$.
	\end{corollary}
	\begin{proof}
		Let $\lambda>\kappa$. Let $j:V_0\rightarrow M$ be a $\lambda$ embedding: i.e. elementary, with critical point $\kappa$, $j(\kappa)>\lambda$ and $M^\lambda \subset M$. Let $G$ be the $\PP:=\col(\omega,\lvert \alpha \rvert)$ generic filter used in constructing $V$ (so $V=V_0[G]$). Note that $\PP$ is small compared to $\kappa$, so $j(\PP)=\PP$ and $G$ is also generic over $M$. Clearly for all $p\in G$, $j(p)=p\in G$. So $j$ extends to an elementary embedding $j^*: V_0[G]\rightarrow M[G]$ by the first lemma. Since $j^*$ extends $j$ it has critical point $\kappa$ and sends $\kappa$ up above $\lambda$. Since $\PP$ is small compared to $\kappa$ it certainly satisfies the $\lambda^+$ chain condition, and therefore by the second lemma $V=V_0[G]$ believes that $M[G]^\lambda \subset M[G]$.
	\end{proof}
	
	From now on, we will forget about $V_0$ and the collapsing forcing, and just work in $V$, a universe where $\alpha$ is countable, GCH holds and $\kappa$ is supercompact, and where $\Reg_{<\alpha}$ is unbounded. Note that since $\kappa$ is supercompact it is hyperinaccessible, and therefore $\Reg_{<\alpha}$ is also unbounded below $\kappa$. We will also assume (by cutting off the top of the model if necessary) that $\Reg_{\alpha}\setminus \kappa^+=\emptyset$. Note that $\kappa$ will be an inaccessible limit of elements of $\Reg_{<\alpha}$ (so it is in $\Reg_{\beta}$ for some $\beta>\alpha$), and by the assumption we just made it is the largest element of this $\Reg_{\beta}$.
	
	\subsection{The Structure of the Proof}
	
	We will construct a forcing which will make $\kappa$ the first element of $\Reg_{\alpha}$ and ensure $\LST(I,Q^\alpha),\LST(I,R^\alpha)\leq \kappa$. We'll start by giving an informal sketch of how the proof is going to work, before we dive into the technicalities.
	
	The forcing $\PP$ we use is built using three components. The first we will meet is a non-Mahlo forcing $\NM_\lambda$ which adds a club $C$ of singular cardinals below some cardinal $\lambda$. The second is a collapsing forcing $\col_\lambda$ which collapses all the cardinals below that $\lambda$, except for those which are a ``short distance'' above some element of $C$. This will leave $\lambda$ as the first element of $\Reg_{\alpha}$.
	
	The third component we meet is a Prikry-style forcing $\QQ_\lambda$, for a cardinal $\lambda$. It adds an $\NM_\lambda$ generic club, adds a new \textit{condition} to the $\col_\lambda$ forcing, and singularises $\lambda$. $\PP$ will be an iteration of $\QQ_\lambda$ forcings (for strategically chosen $\lambda<\kappa$), followed by $\NM_\kappa$ and then $\col_\kappa$. This gives us a universe $V^\PP$ in which $\kappa$ is the first element of $\Reg_\alpha$. How do we show that $\LST(I,Q^\alpha)\leq \LST(I,R^\alpha) \leq \kappa$? We use a similar approach to Theorem \ref{theorem LST maximum}.
	
	Say that $\mathcal{A}$ is a structure with cardinality $\mu\geq\kappa$. Because $\kappa$ is supercompact in $V$, we can find an embedding $j:V\rightarrow M$ such that $j(\kappa)>\mu$ and such that $\mu$ is only a ``short distance" above $\kappa$ (in the sense we were discussing a moment ago). With a good deal of effort, we can show that $\PP$ embeds nicely into $j(\PP)$, and that the conditions of Lemma \ref{lemma j and j*} hold for some suitable $j(\PP)$ generic extension of $M$.
	
	Now, in $M^{j(\PP)}$, we know that no cardinals between $\kappa$ and $\mu$ are collapsed or singularised, because $\mu$ is only a ``short distance'' above $\kappa$ and $\kappa$ is in the club of singulars added by $\NM_{j(\kappa)}$. And below $\kappa$ we can also show that $V^{\PP}$ and $M^{j(\PP)}$ agree about the cardinals and the regular cardinals. Hence, $\mathcal{L}\cup \{I,Q^\alpha\}$ interprets all statements about $\mathcal{A}$ in the same way in $V^{\PP}$ and $M^{j(\PP)}$. (There are some technicalities involved with dealing with $\kappa$ itself here, but they turn out not to be a problem.) From there, the proof follows that of Theorem \ref{theorem LST maximum}.
	
	So in summary, the proof will consist of the following steps:
	
	\begin{enumerate}
		\item Define a function $f$ that formalises the concept of ``a short distance''
		\item Define the non-Mahlo and collapsing forcings $\NM_\lambda$ and $\col_\lambda$ we discussed above
		\item Define the forcing $\QQ_\lambda$, which adds an $\NM_\lambda$ generic filter, adds a condition to $\col_\lambda$ and singularises $\lambda$
		\item Show that $\NM_\lambda * \col_\lambda$ embeds nicely in $\QQ_\lambda * \col_\lambda$
		\item Put these forcings together in some way to get the forcing $\PP$ we will be using
		\item Show that in a $\PP$ generic extension, the $\LST$ is number at most $\kappa$
	\end{enumerate}
	
	\subsection{Defining $f$}
	
	Recall the Laver function, for any supercompact cardinal $\kappa$:
	
	\begin{lemma*}\cite{laverMakingSupercompactnessIndestructible}
		There is a function $h:\kappa \rightarrow V_\kappa$ such that given any $x\in V$ and any $\mu \geq \kappa$, there is an $M$ with  $M^{\mu}\subset M$ and an embedding $j:V\rightarrow M$ with critical point $\kappa$, such that $j(\kappa)>\mu$ and $j(h)(\kappa)=x$.
	\end{lemma*}
	
	We use this $h$ to define a related function $f$ specifically for the ordinals.
	
	\begin{lemma}\label{lemma defining f}
		There is a strictly increasing function $f:\kappa\rightarrow \kappa$, whose range is a subset of $\Reg_\alpha$, such that:
		\begin{enumerate}
			\item For all $\gamma<\kappa$, the interval $(\gamma,f(\gamma))$ contains unboundedly many elements of $\Reg_\beta$ for every $\beta<\alpha$, but no elements of $\Reg_{\alpha}$
			\item For any $\mu>\kappa$, there is a model $M$ with $M^{\mu}\subset M$ and an elementary embedding $j:V\rightarrow M$ with critical point $\kappa$ such that $j(\kappa)>\mu$ and $j(f)(\kappa)>\mu$.
		\end{enumerate}
	\end{lemma}
	\begin{proof}For $\gamma<\kappa$, let $g(\gamma):=h(\gamma)$ if the latter is an ordinal greater than $\gamma$ but smaller than $\kappa$, and $(\gamma,g(\gamma) ) \cap \Reg_{\alpha}=\emptyset$, and otherwise let $g(\gamma):=\gamma$. Let $f(\gamma)$ be the least element of $\Reg_{\alpha}$ which is (strictly) above $g(\gamma)$. It is trivial to see that $f$ is strictly increasing and satisfies the first requirement.
		
		Fix $\mu>\kappa$. Using the previous result, let $M$ and $J:V\rightarrow M$ be such that $j(h)(\kappa)=\mu$. We know that $M$ is closed under $\mu$ sequences, so it correctly calculates the cardinalities and cofinalities of all cardinals below $\mu$. In particular, $\Reg_\alpha^M$ agrees with $\Reg_\alpha^V$ up to $\mu$, and hence $(\kappa,\mu)\cap \Reg_\alpha=\emptyset$ from the perspective of $M$. Hence $j(g)(\kappa)=j(h)(\kappa)=\mu$. So $j(f)(\kappa)>\mu$.
		
	\end{proof}
	
	In the language we used in the preamble to this proof, $\delta$ is ``a short distance" above $\gamma$ if $\delta \in (\gamma,f(\gamma))$.
	
	\subsection{NM and Col}
	Now we have $f$, we can define the simpler forcings used in this construction, $\NM_\lambda$ and $\col_\lambda$. Throughout this section, let $\tilde{V}$ be any universe which contains the $f$ we found in the previous section. We will also assume that $f$ has the same relation to $\Reg$ in $\tilde{V}$ as it does in $V$ (although we do \textit{not} assume that $\Reg_\beta^{\tilde{V}}=\Reg_\beta^V$ for any particular $\beta\leq \alpha$), and (in preparation for the future) we assume that GCH holds in $\tilde{V}$ except that there are perhaps a few singular cardinals $\lambda$ such that $2^\lambda=\lambda^{++}$. Let us fix $\lambda\leq \kappa$ to be a $\lambda^+$ supercompact cardinal in $\tilde{V}$.
	
	
	A Mahlo cardinal is one for which the class of regulars below it is stationary, so to stop it being Mahlo we need to add a club which consists only of singular cardinals. We use a variant of the usual club shooting forcing which also chooses some elements of $\Reg_\alpha$.
	
	\begin{definition}
		We define the non-Mahlo forcing $\NM_\lambda$ as follows. Its elements are closed bounded subsets of $\lambda$, whose minimum element is $\omega$, whose successors are all taken from $\Reg_{\alpha}$, and whose limits are all singular. We order $\NM_\lambda$ by end inclusion.
	\end{definition}
	
	Note that this is not quite the usual non-Mahlo forcing: it gives us a club whose \textit{limits} are singular, but whose successors are elements of $\Reg_{\alpha}$. (Since $\Reg_\alpha$ is unbounded below $\lambda$, we do get an unbounded subset of $\lambda$ despite the odd requirement for successors.) By contrast, a standard non-Mahlo forcing would just add a club of singular cardinals. Of course, we can obtain a fully singular club here simply by deleting all the successors, so $\NM_\lambda$ does indeed force that $\lambda$ is no longer Mahlo.
	
	Although $\NM_\lambda$ is not strictly $<\lambda$ closed (since the limit of a sequence of conditions could be inaccessible) it is almost $<\lambda$ closed, and we can prove the usual results of $<\lambda$ closed-ness.
	
	\begin{lemma}
		Let $\mu<\lambda$. Then there are densely many conditions $p$ such that $\NM_\lambda {\upharpoonleft} p$ is $\mu$ closed.
	\end{lemma}
	\begin{proof}
		Let $p\in \NM_\lambda$ be any condition whose final element is larger than $\mu$. Let $p\geq p_0\geq p_1\geq \ldots$ be a descending sequence of conditions of length $\mu$. Then $p_0 \subset p_1 \subset \ldots$ are all end extensions of one another, so $\bigcup_{i<\mu}p_i$ is a set whose successors are all in $\Reg_\alpha$ and whose limit \textit{elements} are all singular. It is closed, except that it might not contain its supremum $\gamma$. But $\gamma\geq \sup p>\mu$ has cofinality $\cof(\mu)$, so it is singular. This implies that $\gamma<\lambda$, and that $\bigcup_{i<\mu}p_i \cup \{\gamma\}\in \NM_\lambda$.
	\end{proof}
	
	\begin{corollary}\label{corollary basic facts about NM}
		$\NM_\lambda$ is $<\lambda$ distributive, does not collapse or singularise any cardinals, and preserves GCH, except perhaps that $2^\lambda$ becomes $\lambda^{++}$.
	\end{corollary}

	Throughout this whole section, we shall write $C$, and variants thereof, to refer to the generic club added by $\NM_\lambda$. It should (hopefully) always be clear from context to which $\lambda$ we are referring. Notice that for any $\gamma\in C$, we know $\Succ(\gamma)\geq f(\gamma)$, because there are no $\Reg_{\alpha}$ between $\gamma$ and $f(\gamma)$.
	
	Next, we define the collapsing forcing. Recall that if $\gamma$ is regular and $\delta$ is strongly inaccessible, the forcing $\col(\gamma,<\delta)$ collapses all the cardinals between $\gamma$ and $\delta$ down to $\gamma$.
	
	For the rest of this section, let us fix a club $C$ which is generic for $\NM_\lambda$. We will work in some generic extension $\tilde{V}[G]$ of $\tilde{V}$ containing $C$, but we \textit{do not} assume that $\tilde{V}[G]=\tilde{V}[C]$. We will, however, assume that $G$ does not collapse or singularise any cardinals below $\lambda$, and preserves GCH below $\lambda$. In particular, this means that all the successors of $C$ will be strongly inaccessible in $\tilde{V}[G]$. The reason for this slightly arcane approach is that we will eventually need to have definitions for $\col_\lambda$ in two different universes: $\tilde{V}^{\NM_\lambda}$, and $\tilde{V}^{\QQ_\lambda}$. We have already seen that $\NM_\lambda$ preserves GCH; once we define $\QQ_\lambda$ we will discover that it also preserves GCH below $\lambda$.
	
	The forcing we want to work with is a combination of these collapsing forcings, and is defined in terms of $C$. We want a forcing which will collapse, for each $\gamma \in C$, every cardinal between $f(\gamma)$ and $\Succ_C(\gamma)$. Once we have done this, it will be easy to see that it makes $\lambda$ the first element of $\Reg_{\alpha}$. The forcing we want to use is the Easton product of all the collapsing forcings:
	
	\begin{definition}
		Let $\tilde{V}[G]$ be some generic extension of $\tilde{V}$, which contains an $\NM_\lambda^{\tilde{V}}$ generic club $C$, and preserves GCH below $\lambda$. We define the forcing $\col_\lambda$ to be the set of all the elements $h$ of
		$$\prod_{\gamma\in C} \col(f(\gamma),<\Succ_C(\gamma))$$
		
		such that for all regular cardinals $\mu \leq \lambda$, the set
		$$\{\gamma \in C\cap \mu : h(\gamma)\neq \emptyset\}$$
		
		is bounded in $\mu$. We order $\col_\lambda$ in the obvious way.
	\end{definition}

Note that $\col_\lambda$ depends critically on our choice of $\NM_\lambda$ generic club $C$. So really, we should include $C$ as a subscript. But there will never be more than a single generic club below $\lambda$, so we will leave this implicit.

$\col_\lambda$ also depends on which cardinals are regular. So even if two universes both contain the same generic club $C$, we may find that $\col_\lambda$ is different. In particular, once we have defined $\QQ_\lambda$ we will find that $\col_\lambda$ will be different in $\tilde{V}^{\NM_\lambda}$ and $\tilde{V}^{\QQ_\lambda}$, even if both generic extensions add the same club.
	
	If $c$ is a closed subset of $\lambda$, then we extend the above definition by defining $\col(c)$ in the obvious way, as the Easton support product of $\col(f(\gamma),<\Succ_c(\gamma))$ forcings for $\gamma\in c \setminus \max(c)$. If $c\in \tilde{V}$ then this can be defined in the ground model $\tilde{V}$, before we add an $\NM_\lambda$ generic club $C$. Of course, once we do add $C$, we know that if $c$ is an initial segment of $C$ then $\col(c)$ will be a subset of $\col(C)=\col_\lambda$. (Technically, this is only true if we allow a minor abuse of notation about the domains of the functions in $\col(c)$ and $\col_\lambda$.)
	
	The following results are proved in standard ways:
	
	\begin{proposition}\label{proposition col(c) closed}
		Let $c$ be a closed set of cardinals (in $\tilde{V}$ or $\tilde{V}[G]$) whose successors are all inaccessibles below $\kappa$. If $\gamma=\min(c)$ then $\col(c)$ is $<f(\gamma)$ closed in any universe where this makes sense. In particular, this means it is $\gamma^+$ closed.
	\end{proposition}

	\begin{lemma}\label{lemma basic facts about col}
		In any universe where $\col_\lambda$ is defined (i.e. any universe containing an $\NM_\lambda$ generic club $C$ whose successors are strongly inaccessible), it does not collapse any cardinals except those which it is supposed to (i.e. those in the interval $(f(\gamma), \Succ_C(\gamma))$ for some $\gamma \in C$). Nor does it singularise any other cardinals.
	\end{lemma}
	
	\begin{corollary}
		In any universe where $\col_\lambda$ is defined, if $0<\beta<\alpha$ then $\col_\lambda$ does not modify $\Reg_\beta$, other than removing those elements which are in an interval $(f(\gamma),\Succ_C(\gamma))$.
	\end{corollary}
		
		
		
		
	
	Notice also that in the $\col_\lambda$ generic extension, there are no elements of $\Reg_\alpha$ below $\lambda$, and that $\lambda$ is a limit of elements of $\Reg_\beta$ for all $\beta<\alpha$. So if $\lambda$ is regular in $\tilde{V}[G]$ (and hence in the $\col_\lambda$ generic extension) then it will be the first element of $\Reg_\alpha$.
	
	\subsection{The Prikry-style forcing $\mathbb{Q}$}
	
	Again, we fix a $\lambda^+$ supercompact $\lambda\leq\kappa$ in a universe $\tilde{V}$ which knows about $f$.
	
	The forcing $\QQ_\lambda$ is similar to a Prikry forcing, but with some extra components. We want any $\QQ_\lambda$ generic filter to not only singularise $\lambda$, but also to define a $\NM_\lambda$ generic club $C$. For reasons that will become clearer later, we also want it to define some kind of ``generic \textit{element}" of $\col_\lambda$.
	
	Where a standard Prikry forcing would add an $\omega$ sequence of single ordinals, we will arrange for $\QQ_\lambda$ to add an $\omega$ sequence whose terms are taken from the following set:
	
	\begin{definition}
		$K_\lambda$ is the set of all triples $\langle c,x,\epsilon\rangle$ where:
		\begin{enumerate}
			\item $c\in \NM_\lambda$;
			\item $h\in \col(c)$;
			\item $\epsilon \in \On$ and $\max(c)<\epsilon<\lambda$.
		\end{enumerate}
		
		For $\beta<\lambda$, $K_\lambda^\beta$ is the subset of $K_\lambda$ consisting of all the triples $\langle c,h,\epsilon \rangle$ where the least element of $c$ is greater than $\beta$ (and thus also $\epsilon>\beta$).
	\end{definition}
	
	The Prikry-style generic sequence we're aiming for will consist of an element $\langle c_0,h_0,\epsilon_0 \rangle$ of $K_\lambda$, then an element $\langle c_1,h_1,\epsilon_1 \rangle$ of $K_\lambda^{\epsilon_0}$, and so on up through the all $n\in\omega$. From such a sequence $G$, we will be able to extract a club $C=\bigcup_{n\in \omega} c_n$ in $\lambda$, an element $H=\bigcup_{n\in \omega} h_n$ of $\prod_{\gamma\in C} \col(\gamma,<\Succ_C(\gamma))$ and an $\omega$-sequence $(\epsilon_n)$. If we set up the forcing correctly, we will later discover that $C$ is $\NM_\lambda$ generic, that $H\in \col^{\tilde{V}[G]}_\lambda$, and that $(\epsilon_n)$ is cofinal in $\lambda$.
	
	It's fairly easy to see what the finite stem should look like in this context. But how can we find an analogue of a measure $1$ set of ordinals? We must define a measure on $K_\lambda$, and in fact on $K_\lambda^\beta$ for every $\beta<\lambda$. To do this, we first define an ordering on $K_\lambda$:
	
	\begin{definition}
		Let $\langle c,h,\epsilon\rangle$ and $\langle \tilde{c},\tilde{h},\tilde{\epsilon}\rangle$ be elements of $K_\lambda$. We say that $\langle\tilde{c},\tilde{h},\tilde{\epsilon}\rangle\leq^* \langle c,h,\epsilon\rangle$ if:
		
		\begin{enumerate}
			\item $\tilde{c}\leq c$, that is, $\tilde{c}$ is an end extension of $c$;
			\item $\tilde{h}{\upharpoonleft} \max(c) = h$;
			\item $\tilde{\epsilon}\geq \epsilon$.
		\end{enumerate}
	\end{definition}
	
	Notice the unusual nature of the second clause. It's not enough that $\tilde{h}\leq h$ in the $\col(\tilde{c})$ ordering. It must actually agree completely with $h$, although it is allowed to add more things once we're above the area where $h$ is defined. This is because $H=\bigcup h_n$ is supposed to be a \textit{condition} in $\col_\lambda$, not a $\col_\lambda$ generic filter.
	
	Of course, this also defines an ordering on $K_\lambda^\beta\subset K_\lambda$ for every $\beta<\lambda$.
	
	\begin{lemma}
		Let $\beta<\lambda$. Let $F$ be the family of all subsets of $K_\lambda^\beta$ which contain a $\leq^*$ dense open subset of $K_\lambda^\beta$. Then $F$ is a $\lambda$ complete filter in $K_\lambda^\beta$ in the usual $\subset$ ordering of $\mathcal{P}(K_\lambda^\beta)$.
	\end{lemma}
	
	The following is a standard consequence of $\lambda^+$ supercompactness:
	
	\begin{lemma}\cite[22.17]{kanamori}
		Let $S$ be a set of size $\lambda$, and let $F$ be a $\lambda$ complete filter on $\mathcal{P}(S)$. Then $F$ can be extended to a $\lambda$ complete ultrafilter $U$.
	\end{lemma}
		
		
	
	\begin{corollary}
		Let $\beta<\lambda$. There is an ultrafilter $U_\beta$ on $K_\lambda^\beta$ which contains all the $\leq^*$ dense open subsets of $K_\lambda^\beta$.
	\end{corollary}
	
	In fact, of course, there will be many such ultrafilters, but we will fix a single one for the rest of this section to call $U_\beta$. With $U_\beta$ in hand, we can define the analogue of the measure $1$ set of ordinals in a Prikry forcing.
	
	\begin{definition}
		Let $T$ be a tree of height $\omega$, whose nodes are all elements of $K_\lambda$. We abuse notation by allowing the same element of $K_\lambda$ to appear multiple times, provided no element appears twice as direct successors of the same node. We say $T$ is \textit{nice} if the following hold:
		
		\begin{enumerate}
			\item If $\langle \tilde{c},\tilde{h},\tilde{\epsilon}\rangle <_T \langle c,h,\epsilon\rangle\in T$, then $\langle \tilde{c},\tilde{h},\tilde{\epsilon}\rangle\in K_\lambda^\epsilon$;
			\item If $\langle c,h,\epsilon \rangle \in T$ then the set of its direct successors (which is a subset of $K_\lambda^\epsilon$ by the previous condition) is in $U_\epsilon$.
		\end{enumerate}
	\end{definition}
	
	In the forcing $\QQ_\lambda$, our conditions will have two components: a finite sequence $s$ of terms from $K_\lambda$ and a nice tree $T$ which gives us a map of where the sequence is allowed to go from there. The root of $T$ will be the final term of $s$.
	
	\begin{definition}
		The forcing $\QQ_\lambda$ has conditions of the form
		
		$$\big((\langle c_0,h_0,\epsilon_0\rangle, \langle c_1,h_1,\epsilon_1\rangle,\ldots,\langle c_{n-1},h_{n-1},\epsilon_{n-1}\rangle),T\big)$$
		
		for some $n\in \omega$, where
		
		\begin{enumerate}
			\item $\langle c_0,h_0,\epsilon_0\rangle \in K_\lambda$ if $n>0$;
			\item For $0<i<n$, $\langle c_i,h_i,\epsilon_i\rangle \in K_\lambda^{\epsilon_{i-1}}$;
			\item $T$ is a nice tree
			\item The root of $T$ is $\langle c_{n-1},h_{n-1},\epsilon_{n-1}\rangle$.
		\end{enumerate}
		
		We call $\epsilon_{n-1}$ the height of the condition (writing $\height$ in symbols).
		
		The conditions are ordered in the usual way for a Prikry style forcing: $(s',T') \leq (s,T)$ if $s'$ is an end extension of $s$ and there is a path $B=b_0,b_1,\ldots,b_k$ through $T$ of length $k:=\lvert s'\setminus s\rvert+1$ such that:
		\begin{enumerate}
			\item $b_0$ is the root of $T$ (i.e. the last element of $s$)
			\item For all $0<i\leq k$, $b_i$ is the $i$'th element of $s'\setminus s$
			\item $T'$ is a subtree of $T$ whose root is $b_k$
		\end{enumerate}
		
		As usual for Prikry forcings, if $s'=s$ we say $(s',T')$ is a direct extension of $(s,T)$ and write $(s',T')\leq^*(s,T)$.
	\end{definition}
	
	Note: We now have two different definitions of $\leq^*$. One talks about elements of $K_\lambda$ and the other about conditions in $\QQ_\lambda$.
	
	\begin{proposition}\label{proposition Qlambda quasi closed}
		The forcing $(\QQ_\lambda, \leq^*)$ is ${<}\lambda$ closed. That is, any descending sequence $T_0\supset T_1\supset T_2\ldots$ of nice trees with the same root will have a nice tree as their intersection.
	\end{proposition}
	\begin{proof}
		Follows from the fact that $U_\epsilon$ is closed under $<\lambda$ intersections.
	\end{proof}
	
	\noindent From this, we can prove the Prikry property.
	
	\begin{lemma}\label{Lemma Prikry Property of Q}
		Let $p\in \QQ_\lambda$. Let $\varphi$ be first order (perhaps with parameters). Then there is some $q\leq^* p$ deciding $\varphi$.
	\end{lemma}
The proof is left as an exercise for the reader. It's essentially the same as the standard proof for Prikry forcings, but the notation gets rather technical when it is written out formally.

	\begin{corollary}
		Let $\varphi(x)$ be a formula with one free variable, let $\mu\leq\lambda$, and suppose $p\forces \exists x<\check{\mu} \,\varphi(x)$. Then there is some $q\leq^* p$ and some $\gamma <\mu$ such that $q\forces \varphi(\check{\gamma})$.
	\end{corollary}

\begin{proof}
	Standard from \ref{Lemma Prikry Property of Q}.
\end{proof}
	
	This tells us that $\QQ_\lambda$ does not do any unexpected collapsing or singularising of cardinals.
	
	\begin{lemma}\label{lemma facts about Qlambda}
		$\QQ_\lambda$ does not add any new bounded subsets of $\lambda$, collapse any cardinals, or singularise any cardinals apart from $\lambda$. It preserves GCH where it holds in $\tilde{V}$, except at $\lambda$.
	\end{lemma}

	\subsection{Embedding NM*Col into Q*Col}
	
	Once again, fix a $\lambda^+$ supercompact $\lambda\leq \kappa$ in a universe $\tilde{V}$ which knows about $f$. 
	As we discussed earlier, we can extract elements of $\NM_\lambda$ and $\col_\lambda$ from a $\QQ_\lambda$ condition.
	
	\begin{definition} We define two schemes of abbreviations.
		
		\begin{enumerate}
			\item Let
			$$p=((\langle c_0,h_0,\epsilon_0\rangle, \langle c_1,h_1,\epsilon_1\rangle,\ldots,\langle c_{n-1},h_{n-1},\epsilon_{n-1}\rangle),T)$$
			be a condition in $\QQ_\lambda$. We define $c_p=\bigcup_{i<n} c_i$ and $h_p=\bigcup_{i<n}h_i$.
			\item Let $G$ be $\QQ_\lambda$ generic. We define $C_G=\bigcup_{p\in G} c_p$ and $H_G=\bigcup_{p\in G} h_p$.
		\end{enumerate}
	\end{definition}
	
	An equivalent, but less friendly, definition is that
	
	$$C_G=\bigcup \{c: \exists h, \epsilon \,\, \langle c,h,\epsilon\rangle \text{ is a term in the first part of some } p\in G\}$$
	and that
	$$H_G=\bigcup \{h: \exists c, \epsilon \,\, \langle c,h,\epsilon\rangle \text{ is a term in the first part of some } p\in G\}$$
	
	We shall now establish what these four objects actually are, in terms of $\NM_\lambda$ and $\col_\lambda$. The first two objects, $c_p$ and $h_p$, are simply conditions of the relevant forcings in $\tilde{V}$:
	
	\begin{proposition}\label{proposition c_p h_p}
		Let $p\in \QQ_\lambda$. Then $c_p\in \NM_\lambda$ and $h_p \in \col(c_p)$ in $\tilde{V}$.
	\end{proposition}
		
		
		
	
	$C_G$ will be generic for $\NM_\lambda$, and there is a useful correspondence between $\NM_\lambda$ and $\QQ_\lambda$:
	
	\begin{lemma}
		If $G$ is $\QQ_\lambda$ generic, then $C_G$ is a club which is $\NM_\lambda$ generic (over $\tilde{V}$). Moreover, given any $\NM_\lambda$ name $\sigma$, there is a $\QQ_\lambda$ name $\varphi(\sigma)$ such that for any $\QQ_\lambda$ filter $G$ (not necessarily generic), $\sigma^{C_G}=\varphi(\sigma)^{G}$.
	\end{lemma}
	(As usual, in the statement of the lemma we are muddling the definition of a generic filter over $\NM_\lambda$, and the club corresponding to that generic filter.)
	\begin{proof}
		It is easy to see that $C_G$ is a club in $\lambda$, and that all its closed initial segments are elements of $\NM_\lambda$. We must show that it is generic. Let $D\subset \NM_\lambda$ be open dense. We will show that $X_D:=\{p\in \QQ_\lambda: c_p\in D\}$ is dense in $\QQ_\lambda$.
		Fix a condition $p=(s,T)\in \QQ_\lambda$. Since $D$ is dense in $\NM_\lambda$,
		$S_{D,p}:=\{\langle c,h,\epsilon\rangle \in K_\lambda^{\height(p)} : c_p\cup c\in D\}$ is $\leq^*$ dense in $K_\lambda^{\height(p)}$. Clearly, it is also open.	So $S_{D,p}\in U_{\height(p)}$ because $U_{\height(p)}$ contains all the open dense subsets of $K_\lambda^{\height(p)}$. Since the set of all successors of the root of $T$ is also in $U_{\height(p)}$, there must be a level $1$ element of the tree $T$ which is in $S_{D,p}$. But then that gives us a $1$ step extension $q\leq p$ in $\QQ_\lambda$, such that the extra term $\langle c,h,\epsilon\rangle$ in $q$ is an element of $S_{D,p}$. It follows that $c_q=c_p\cup c \in D$, and hence $q\in X_D$ as required.
		
		The second part of the lemma is an easy recursion definition.
		
	\end{proof}
	
	So $\QQ_\lambda$ behaves well regarding $\NM_\lambda$. What about $\col_\lambda$? We want $\QQ_\lambda * \col^{\tilde{V}[G]}_\lambda$ to play nicely with $\NM_\lambda * \col^{\tilde{V}[C_G]}_\lambda$. In both cases, $\col_\lambda$ is defined using the same club $C_G$. However, $\col^{\tilde{V}[G]}(C_G)$ and $\col^{\tilde{V}[C_G]}(C_G)$ are not the same: we used an Easton product, and $\lambda$ is regular in $V[C_G]$ but not in $V[G]$. On the other hand, we do at least know that $\col^{\tilde{V}[C_G]}(C_G)\subset \col^{\tilde{V}[G]}(C_G)\cap \tilde{V}[C_G]$.
	
	
	This is where the condition $H_G$ named by $G$ comes in. It is a sort of ``generic element'' of $\col^{\tilde{V}[G]}(C_G)$ which ensures the two forcings will play nicely. First, we must verify that $H_G$ really is a condition.
	
	\begin{lemma}
		Let $G$ be $\QQ_\lambda$ generic. Then $H_G\in \col^{\tilde{V}[G]}_\lambda$.
	\end{lemma}
	\begin{proof}
		It is easy to see that $H_G\in \prod_{\gamma\in C_G} \col(\gamma,<\Succ_{C_G}(\gamma))$. We must verify that its support is bounded below every $\tilde{V}[G]$ regular cardinal $\mu\in[0,\lambda]$. If $\mu<\lambda$, this follows from Proposition \ref{proposition c_p h_p}: take $p\in G$ such that $\sup c_p > \mu$, and then $H_p\vert \mu = h_p \in \col(c_p)$ and hence (since $\mu$ is regular in $\tilde{V}$) the support of $H_p$ is bounded below $\mu$.
		
		On the other hand, we know that $\QQ_\lambda$ singularises $\lambda$, so the case $\mu=\lambda$ is vacuous.
	\end{proof}
	
	\begin{lemma}\label{lemma difficult magidor lemma}\cite[Lemma 30]{magidorVaananan}
		Let $G$ be $\QQ_\lambda$ generic. Let $G^*$ be $\col^{\tilde{V}[G]}_\lambda$ generic over $\tilde{V}[G]$, and contain $H_G$. Then the filter $G^{**}:=G^*\cap \col^{\tilde{V}[C_G]}_\lambda$ is $\col^{\tilde{V}[C_G]}_\lambda$ generic over $\tilde{V}[C_G]$.
	\end{lemma}
	\begin{proof}
		It is easy to check that $G^{**}$ is a filter; the challenge is showing that it's generic. So let $\dot{D}$ be an $\NM_\lambda$ name for a dense open subset of $\col^{\tilde{V}[C_G]}(C_G)=\col^{\tilde{V}[C_G]}_\lambda$. 
		
		For any $\QQ_\lambda$ generic filter $G$, and for any $\col^{\tilde{V}[G]}(C_G)$ generic filter $G^*$ containing $H_G$, the set $G^*\cap D \neq \emptyset$, where $D=\dot{D}^{C_G}$.
		
		To begin with, we work in $\tilde{V}[C_G]$ for a fixed filter $G$. Let $D=\dot{D}^{C_G}$ as above. For $\delta\in \Succ(C_G)$, consider two forcings $\PP^\delta:=\col^{\tilde{V}[C_G]}(C_G\setminus \delta+1)$ and  $\PP_\delta:=\col^{\tilde{V}[C_G]}(C_G\cap (\delta+1))$. Let
		
		For $h\in \PP^\delta$ we define the set $D_h:=\{h'\in \PP_\delta: h'\cup h\in D\}$. We then define the set ${D^\delta:=\{h\in \PP^\delta: D_h\text{ is open dense}\}}$.
		
		\begin{claim}
			$D^\delta$ is an open dense subset of $\PP^\delta$.
		\end{claim}
		\begin{proof}
			$D$ is open, so if $\tilde{h}\leq h\in \PP^\delta$ then $D_{\tilde{h}}\supset D_h$. Hence $D_{\tilde{h}}$ is dense. Moreover, $D_{\tilde{h}}$ is also open, since $D$ is open. Hence $D^\delta$ is open. To show $D^\delta$ is also dense, let us fix $h\in \PP^\delta$.
			
			Now $\PP_\delta$ has cardinality $\delta$, so we can enumerate its elements $\{h_\gamma: \gamma<\delta\}$. We construct a decreasing sequence $(h_\gamma')_{\gamma<\delta+1}$ of length $\delta$ of conditions in $\PP^\delta$. Let $h_0'=h$. For $\epsilon<\delta$, we choose some $h_{\epsilon+1}'\leq h_\epsilon'$ such that for some $h^*\leq h_\epsilon$, the condition $h^*\cup h_{\epsilon+1}'\in D$. We can do this easily, since $D$ is dense: just take some element of $D$ below $h_\epsilon \cup h_\epsilon'$ and let $h_{\epsilon+1}'$ be the part of it which is above $\delta$. For limit $\gamma\leq \delta$, we take $h_\gamma$ to be below every earlier term of the sequence, which we can do since $\PP^\delta$ is $\delta^+$ closed by Proposition \ref{proposition col(c) closed}.
			
			Now, $h_\delta'\leq h$ is such that for all $h_\gamma\in \PP_\delta$, there exists $h^*\leq h_\gamma$ such that $h^*\cup h_\delta'\in D$, since $h_\delta'\leq h_{\gamma+1}'$ and $D$ is open. It follows immediately that $D_{h_\delta'}$ is dense. Since $D$ is open, it's also immediate that $D_{h_\delta'}$ is open. Hence $h_\delta' \in D^\delta$. But $h'_\delta\leq h$ so $D^\delta$ is dense.
		\end{proof}
		
		We now work in $\tilde{V}$. We shall show that $\mathbbold{1}_{\QQ_\lambda}$ forces the following statement:
		
		\begin{quotation}``There are two ordinals $\delta<\beta$ in $C_G$, with $\delta$ a successor element of $C_G$, such that $H_G{\upharpoonleft} [\delta,\beta]\in D^\delta$.''\end{quotation}
		
		This statement makes sense, since $\tilde{V}[C_G]$ is a definable subclass of $\tilde{V}[G]$  and so $\tilde{V}[G]$ knows what $D^\delta$ looks like for any $\delta$. 
		Fix $p\in \QQ_\lambda$. Let $\height(p)=\epsilon$. We shall show there is a one step extension of $p$ which forces the above statement.
		
		Let us say $\langle c,h,\gamma\rangle \in K_\lambda^\epsilon$ is \textit{cooperative} if there is some element $\delta \in \Succ(c)$ such that ${c_p\cup c \Vdash_{\NM_\lambda} h{\upharpoonleft} (c\setminus \delta)\in D^\delta}$. (Recall that by definition of $K_\lambda$, we know $h\in \col^{\tilde{V}}(c)$, so $h\in \col^{\tilde{V}[C_G]}(C_G)$. So $h{\upharpoonleft} (c\setminus \delta)\in \PP^\delta$ in $\tilde{V}[C_G]$, and thus it makes sense ask whether it is in $D^\delta$.)
		
		\begin{claim}The set of all cooperative elements of $K_\lambda^\epsilon$ is $\leq^*$ dense.
		\end{claim}
		\begin{proof}
			Let $\langle c,h,\gamma\rangle \in K_\lambda^\epsilon$. Without loss of generality, we can assume $c$ has a largest element. (If it doesn't, then we can simply extend $c$ arbitrarily by one step to get a $(c',h,\gamma')<^*(c,h,\gamma)$ and work with that instead.) Let $\delta$ be that largest element.
			
			$c_p\cup c$ forces that $\delta$ is a successor element of the club that $\NM_\lambda$ adds, so it forces that $D^\delta$ is an open dense subset of $\PP^\delta$. Hence, there is some end extension $c'\leq c$ and some name $\dot{h}'$ for an element of $\PP^\delta$ such that $c_p\cup c'\Vdash_{\NM_\lambda} \dot{h}'\in D^\delta$. 
			
			Now, in $\tilde{V}[C_G]$ an element of $\PP^\delta$ is a sequence of conditions in collapsing forcings whose support is bounded below $\lambda$. Each of these collapsing forcings also has cardinality less than $\lambda$. So any element of $\PP^\delta$ in $\tilde{V}[C_G]$ can be coded as a subset of $\lambda$. $\NM_\lambda$ is $\lambda$ distributive, so all the elements of $\PP^\delta$ in $\tilde{V}[C_G]$ actually already existed in $\tilde{V}$.
			
			So without loss of generality, we can choose $\dot{h}'$ to be a check name for some $h'\in \tilde{V}$. Now $h'$ is certainly bounded below $\lambda$, so without loss of generality we may assume that $c'$ is longer than the support of $h'$, that is, that $h'\in \col^{\tilde{V}}(c')=\col^{\tilde{V}[C_G]}(c')$. In fact, since $h'\in P^\delta$, we know $h'\in \col(c'\setminus \delta)$. In particular, since $h \in \col(c)$ and $\sup(c)=\delta$, we know $h\cup h'$ is a well defined element of $\col(c')$. Take $\gamma'\geq \gamma$ to be some ordinal which is larger than $\sup (c')$.
			
			Then $\langle c',h\cup h',\gamma'\rangle \in K_\lambda^\epsilon$ and $\langle c',h\cup h',\gamma'\rangle\leq^*\langle c,h,\gamma\rangle$. Using the same $\delta$ as above, we can see by construction that $c_p\cup c' \Vdash_{\NM_\lambda} (h\cup h') {\upharpoonleft} (c'\setminus \delta)=h'\in D^\delta$. Hence $(c',h\cup h',\gamma')$ is cooperative.
		\end{proof}
		
		So the set $\tilde{K}$ of all cooperative elements of $K_\lambda^\epsilon$ is in $U_\epsilon$. Since the set of all valid one step extensions of $p$ is also in $U_\epsilon$, there is a cooperative $\langle c,h,\gamma\rangle $ which is a valid way to extend $p$ by one step. Let $q$ be this one step extension of $p$. Since $\langle c,h,\gamma\rangle$ is cooperative, we can find $\delta\in \Succ(c)$ such that $c_p\cup c \Vdash_{\NM_\lambda} h{\upharpoonleft} (c\setminus \delta)\in D^\delta$. Let $\beta=\sup(c)$.
		
		Now, let $G$ be a $\QQ_\lambda$ generic filter with $q\in G$. Then $H_G{\upharpoonleft} [\delta,\beta]=h{\upharpoonleft} (c\setminus \delta)$, and $c_p\cup c \in C_G$. So $H_G{\upharpoonleft} [\delta,\beta]\in D^\delta$. Hence $q$ forces ``There are two ordinals $\delta<\beta$ in $C_G$, with $\delta$ a successor element of $C_G$, such that $H_G{\upharpoonleft} [\delta,\beta]\in D^\delta$''. The condition $p$ was arbitrary, so $\mathbbold{1}_{Q_\lambda}$ forces the statement, which is what we wanted to show.
		
		Now let $G$ be an arbitrary $\QQ_\lambda$ generic filter $G$, and work over $\tilde{V}[G]$. The statement is true of $\tilde{V}[G]$, so find $\delta<\beta$ that fit it. Let $G^*$ be a $\col^{\tilde{V}[G]}(C_G)$ generic filter containing $H_G$. Recall that we are aiming to show that $G^*\cap D=G^*\cap \dot{D}^{C_G}\neq \emptyset$. Let $h=H_G{\upharpoonleft} [\delta,\beta]$. So $h\in D^\delta$. Hence, $D_h$ is open and dense in $\PP_\delta=\col^{\tilde{V}[C_G]}(C_G\cap(\delta+1))$. Also, $h\in G^*$ since $H_G\leq h$.
		
		Now, $\QQ_\lambda$ adds no new bounded subsets of $\lambda$ by Lemma \ref{lemma facts about Qlambda}, so $\col^{\tilde{V}}(C_G\cap (\delta+1))=\col^{\tilde{V}[C_G]}(C_G\cap(\delta+1))=\col^{\tilde{V}[G]}(C_G\cap(\delta+1))$. So $D_h$ is open dense over the latter forcing. Since $G^*$ is generic over $\col^{\tilde{V}[G]}(C_G)$, the restriction of $G^*$ to $\col^{\tilde{V}[G]}(C_G\cap (\delta+1))$ is generic over that forcing. Hence $G^*\cap D_h\neq \emptyset$, so it contains some $h'$. But then $h'\cup h\in D$ by definition of $D_h$ (and hence also $h'\cup h\in \col^{\tilde{V}[C_G]}(C_G)$). Since $h',h\in G^*$ also $h'\cup h\in G^*$. So $G^{**}\cap D=G^*\cap \col^{\tilde{V}[C_G]}(C_G)\cap D$ contains $h'\cup h$ and is therefore nonempty.
		
		So fixing a $\QQ_\lambda$ generic $G$, we know that if $G^*$ is any $\col^{\tilde{V}[G]}(C_G)$ generic filter containing $H_G$ and $D$ is any $\col^{\tilde{V}[C_G]}(C_G)$ dense set in $\tilde{V}[C_G]$, then $G^{**} := G^*\cap \col^{\tilde{V}[C_G]}(C_G)$ meets $D$. Hence this $G^{**}$ is generic over $\tilde{V}[C_G]$, concluding the proof of the lemma.
	\end{proof}
	
	As a consequence of our earlier results about $\NM$, $\QQ$ and $\col$, we can also easily conclude that the generic extensions agree about $\Card$ and $\Reg_\beta$ for all $\beta<\alpha$
	
	\begin{lemma}\label{lemma Qlambda, NMlambda, Col don't collapse/singularise badly}
		Let $G$ be $\QQ_\lambda$ generic, let $G^*$ be $\col^{\tilde{V}[G]}_\lambda$ generic and contain $H_G$, and let $G^{**}$ be obtained from $G^*$ as in lemma \ref{lemma difficult magidor lemma}. The cardinals of $\tilde{V}[G][G^*]$ and $\tilde{V}[C_G][G^{**}]$ are both precisely the cardinals of $\tilde{V}$ which are not in the interval $(f(\gamma),\Succ_C(\gamma))$ for any $\gamma \in C$. Moreover, in both these generic extensions, for all $0<\beta<\alpha$, the class $\Reg_\beta$ in the generic extension is precisely $\Reg_\beta^{\tilde{V}}$ but with the intervals $(f(\gamma),\Succ_C(\gamma))$ omitted. In particular, $\Reg_\beta^{\tilde{V}}$ and $\Card$ agree up in $\tilde{V}[G][G^*]$ and $\tilde{V}[C_G][G^{**}]$, for all $\beta<\alpha$.
	\end{lemma}
	
	
	\subsection{Iterating the Q forcings}
	
	We are finally ready to put together these forcings, and define the overall forcing we're going to be using. We use a Prikry style iteration of $\QQ_\lambda$ forcings, followed by an $\NM_\kappa$ forcing and the corresponding $\col_\kappa$ forcing. This will give us a universe in which $\kappa$ is the first element of $\Reg^\alpha$, and is also $\LST(I,Q^\alpha)$. We now drop our discussions of $\tilde{V}$, and just work in the universe $V$ we fixed near the start of the proof. Recall that $V$ believes GCH.
	
	
	\begin{definition}
		We recursively define forcings $P_\gamma$ ($\gamma\leq \kappa)$ with two orders $\leq$ and $\leq^*$, and (for $\gamma<\kappa$) $P_\gamma$ names $\dot{Q}_\gamma$ for forcings, also with two orders $\leq$ and $\leq^*$, as follows.
		\begin{itemize}
			\item For $\gamma\leq\kappa$, the elements of $P_\gamma$ are sequences $\langle \tau_\epsilon\rangle_{\epsilon<\gamma}$ of length $\gamma$, of Easton support, such that for $\delta<\gamma$, the sequence $\langle \tau_\epsilon\rangle_{\epsilon<\delta}\in P_\delta$ and forces $\tau_\delta \in \dot{Q}_\delta$.

			\item $\langle \tau_\epsilon'\rangle_{\epsilon<\gamma}\leq \langle \tau_\epsilon\rangle_{\epsilon<\gamma}$ if:
			\begin{enumerate}
				\item For all $\delta<\gamma$, $\langle \tau_\epsilon'\rangle_{\epsilon<\delta}\Vdash \tau_\delta'\leq \tau_\delta$; and
				\item For all but finitely many $\delta$, either $\langle \tau_\epsilon\rangle_{\epsilon<\delta}\forces\tau_\delta\mathbbold{1}_{\dot{Q}_\delta}$, or we can replace $\leq$ with $\leq^*$ (in the sense of $\dot{Q}_\delta$) on the previous line.\label{Gitik iteration order step}
			\end{enumerate}
			\item $\langle \tau_\epsilon'\rangle_{\epsilon<\gamma}\leq^* \langle \tau_\epsilon\rangle_{\epsilon<\gamma}$ if in the above, \ref{Gitik iteration order step} holds for \textit{every} $\delta$.
			\item For any $\gamma<\kappa$, $\dot{Q}_\gamma$ is a $P_\gamma$ name for a forcing:
			\begin{enumerate}
				\item If $\gamma$ is $\gamma^+$ supercompact in $V$ and $P_\gamma$ forces: ``$\gamma$ is a cardinal which is $\gamma^+$ supercompact", then $\dot{Q}_\gamma$ is a name for the forcing $\QQ_\gamma$ we defined earlier.
				\item Otherwise, $\dot{Q}_\gamma$ is the canonical name for the trivial forcing.
			\end{enumerate}
		\end{itemize}
	\end{definition}

	\begin{lemma}\label{lemma Pkappa properties chain and cardinality}\cite[1.3]{gitik_changingCofinalities} Let $\gamma \leq \kappa$ be a Mahlo cardinal (of $V$). Then $P_\gamma$ has the $\gamma$ chain condition, and has cardinality $\leq\gamma$.
	\end{lemma}
	
	\begin{corollary}\label{corollary Pkappa properties general size}
		Let $\gamma \leq \kappa$. Then $\lvert P_\gamma \rvert \leq \gamma^{++}$.	
	\end{corollary}
	\begin{proof}
		Case 1: $\gamma$ is a Mahlo cardinal. Then $\lvert P_\gamma \rvert \leq\gamma$ by the previous lemma.
		
		Case 2: $\gamma$ is a limit of Mahlo cardinals. Let $(\epsilon_i)_{i<\cf(\gamma)} \subset \gamma$ be a sequence of Mahlo cardinals which is cofinal below $\mu$. By Lemma \ref{lemma Pkappa properties chain and cardinality}, for all $i$, $\lvert P_{\epsilon_i}\rvert \leq \epsilon_i$.	A condition of $P_\gamma$ can be expressed as a collection of conditions, one from each $P_{\epsilon_i}$, which all agree with each other. So
		$$\lvert P_\gamma\rvert \leq \prod_{i<\cf(\gamma)}\lvert P_{\epsilon_i}\rvert\leq \gamma^\gamma =\gamma^+$$
			
		Case 3: $\gamma$ is neither a Mahlo cardinal nor a limit of Mahlo cardinals. Let $\delta<\gamma=\sup(\text{Mahlo} \cap \gamma)$. Any $\lambda^+$ supercompact cardinal $\lambda$ is Mahlo, so we know that $\dot{Q}_\epsilon$ is trivial for all $\epsilon \in (\delta,\gamma)$. So
		$	\lvert P_\gamma\rvert \leq \lvert P_\delta * \dot{Q}_\delta\rvert \leq \delta^{++} < \gamma^{++}$.
	\end{proof}
	
	\begin{lemma}\label{lemma Pkappa properties decidability}\cite[1.4]{gitik_changingCofinalities}
		Let $\varphi$ be a statement (with parameters) and $p\in P_\gamma$. There is some $q\leq^* p$ which decides $\varphi$. The same is true of the forcing $P_\gamma/P_\delta$ if $\delta<\gamma$ is a Mahlo cardinal.
	\end{lemma}
	
	Here $P_\gamma/P_\delta$ is the ($P_\delta$ name for the) forcing defined in the usual way in a $P_\delta$ generic extension $V[G]$: we simply take $P_\gamma$ and delete all those conditions which are incompatible with an element of $G$.
	
	\begin{lemma}\label{lemma Pkappa properties quasi closedness}
		For any $\delta<\gamma\leq \kappa$, the forcing $P_\gamma/P_\delta$ is closed under $\leq^*$ sequences of length less than $\delta$.
	\end{lemma}

The proof uses an adaptation of standard techniques, and is left to the reader.
	
	\begin{corollary}\label{lemma Pkappa properties preservation of small cardinals}
		If $\delta<\gamma$ is a Mahlo cardinal, then the forcing $P_\gamma/P_\delta$ does not collapse or singularise any cardinals below $\delta$, or add any new bounded subsets of $\delta$.
	\end{corollary}
		
		
		
		
	
	\begin{corollary}\label{corollary Pkappa properties GCH}
		If $\gamma$ is a Mahlo cardinal, and $G$ is a $P_\gamma$ generic filter, then GCH holds in $V[G]$ except perhaps at the cardinals $\lambda$ for which $\dot{Q}_\lambda$ is nontrivial. (In other words, if $\lambda$ is \textit{not} $\lambda^+$ supercompact in $V$, or $\lambda\geq \gamma$, then $2^\lambda=\lambda^+$ in $V[G]$.)
	\end{corollary}
	\begin{proof}
		Let $\lambda$ be a cardinal of $V$ which is not $\lambda^+$ supercompact. We first examine the case $\lambda\geq \gamma$. By Lemma \ref{lemma Pkappa properties chain and cardinality} we know that if $\lambda>\gamma^+$ then $2^\lambda=\lambda^+$ in any $P_\gamma$ generic extension. Similarly, if $\lambda=\gamma$ or $\lambda=\gamma^+$, then by regularity we know that for all $\delta<\lambda, \lambda^\delta=\lambda$. So again by the same two lemmas, we know $P_\gamma$ preserves $2^\lambda$.
		
		Now suppose that $\lambda<\gamma$. Let $\mu\leq \lambda=\sup\{\delta\leq \lambda: \delta \text{ is }\delta^+\text{ supercompact}\}$ and let $\nu$ be the least $\delta>\lambda$ such that $\delta$ is $\delta^+$ supercompact. Then up to some trivial notation changes,
		$$P_\gamma=P_\mu * \dot{Q}_\mu * P_\gamma/P_\nu$$
		
		We have just seen that $P_\mu$ preserves $2^\lambda$. If $\dot{Q}_\mu$ is nontrivial, then $\mu$ is $\mu^+$ supercompact in $V$, and hence $\mu<\lambda$. If so, then $\dot{Q}_\mu$ is $\QQ_\mu$, which we saw in Lemma \ref{lemma facts about Qlambda} preserves $2^\lambda$. And finally, $P_\gamma/P_\nu$ adds no new bounded subsets of $\nu>\lambda$, so it doesn't change $\mathcal{P}(\lambda)$ at all. Hence $P_\gamma$ preserves $2^\lambda$.
	\end{proof}
	
	\begin{lemma}\label{lemma Pkappa properties preservation of supercompacts} If $\gamma\leq \kappa$ is $\gamma^+$ supercompact in $V$ then it is still $\gamma^+$ supercompact in the $P_\gamma$ generic extension.
	\end{lemma}
	\begin{proof}
		We adapt the argument of \cite[38]{magidorVaananan}, using Corollary \ref{corollary Pkappa properties GCH} in place of the GCH to show the enumeration of $P_\gamma$ terms is of length $\gamma^{++}$.

	\end{proof}
	
	\begin{corollary}\label{lemma Pkappa properties nontriviality of Q}
		For $\gamma<\kappa$, $\dot{Q}_\gamma$ is nontrivial if any only if $\gamma$ is $\gamma^+$ supercompact in $V$.
	\end{corollary}
	
	\begin{lemma}\label{lemma Pkappa properties preservation of f}
		Let $\gamma\leq \kappa$ and let $G$ be a $P_\gamma$ generic filter. Suppose that $\delta<\gamma$ is in $\Reg_\alpha$ from the perspective of $V[G]$. Then $V[G]$ believes that $(\delta,f(\delta))$ contains unboundedly many elements of $\Reg_\beta$ for all $\beta<\alpha$, but no elements of $\Reg_\alpha$. Also, $V[G]$ does not believe that $f(\delta)$ is in $\Reg_\alpha$.
	\end{lemma}
	\begin{proof}
		Since the interval contains no elements of $\Reg_\alpha$ even in $V$ by definition of $f$, it is immediate that it contains none in $V[G]$ either. Similarly, $f(\delta)\not \in \Reg_\alpha^{V[G]}$. Let $\mu\leq \delta$ be be the supremum of the class of $\lambda^+$ supercompacts $\lambda$ which are $\leq \delta$. Let $\nu>\delta$ be the least $\lambda>\delta$ which is $\lambda^+$ supercompact. So $P_\gamma=P_\mu * \dot{Q}_\mu * (P_\gamma / P_\nu)$.
		
		It follows trivially from Corollary \ref{corollary Pkappa properties general size} that $P_\mu$ satisfies the $\mu^{++}$ chain condition, and therefore does not collapse or singularise any cardinals $\geq \mu^{++}$. But we know that $f(\delta)$ is a limit of elements of $\Reg_0=\Succ$, so $f(\delta)$ is a limit cardinal. Hence, $\mu^{++}\leq \delta^{++}<f(\delta)$.
		
		$\dot{Q}_\mu$ is either trivial, or is of the form $\QQ_\mu$, which we know by Lemma \ref{lemma facts about Qlambda} does not collapse or singularise any cardinals other than $\mu$.
		
		$\nu$ is a successor element of $X$, and hence is $\nu^+$ supercompact in $V$. In particular, then, it is Mahlo. So by Corollary \ref{lemma Pkappa properties preservation of small cardinals}, $P_\gamma / P_\nu$ does not collapse or singularise any cardinals below $\nu$. But since $\nu>\delta>\alpha$, and (again) $\nu$ is Mahlo in $V$, there are certainly elements of $\Reg_\alpha$ in the interval $(\delta,\nu)$. It follows that $\nu>f(\delta)$.
		
		Putting these facts together, we find that $P_\gamma$ does not collapse or singularise any cardinals in the interval $((\delta^{++})^V,f(\delta))$, and hence that in $V[G]$ there are unboundedly many elements of $\Reg_\beta$ for all $\beta<\alpha$.
	\end{proof}

	\subsection{The LST number}
	
	So now we know that $P_\kappa$ is a well-defined forcing which combines a $\QQ_\lambda$ on every $\lambda^+$ supercompact $\lambda<\kappa$. Also, $\kappa$ is still $\kappa^+$ supercompact and hence Mahlo in the generic extension, so the following forcing, the one we will use to get the model of ``$\LST(I,Q^\alpha)=\LST(I,R^\alpha)=\min (\Reg_\alpha)$'', is well-defined:
	
	\begin{definition}
		$$\PP:=P_\kappa * \NM_\kappa * \col_\kappa$$
	\end{definition}
	
	Now, let us fix  a $\PP$ generic extension $V^{**}$ of $V$. Let the corresponding generic filter be $\tilde{G}*C*G^{**}$. So $\tilde{G}$ is a $P_\kappa$ generic filter, and $C$ is an $\NM_\kappa$ generic club over $V[\tilde{G}]$, and $G^{**}$ is a $\col_\kappa$ generic extension over $V[\tilde{G}][C]$.
	
	
	To tidy up our notation, let us write $\tilde{V}$ for $V[\tilde{G}]$. By Corollary \ref{corollary basic facts about NM} and Lemma \ref{lemma basic facts about col} we know that the cardinals of $V^{**}$ below $\kappa$ are precisely the cardinals of $\tilde{V}$ which are either below $\min(C)=\omega$ or are in some interval $[\gamma,f(\gamma)]$ for some $\gamma \in C$. By Lemma \ref{lemma Pkappa properties preservation of f} it follows that in $V^{**}$ there are unboundedly many elements of $\Reg_\beta$ below $\kappa=\sup(C)$ for all $\beta<\alpha$, but no elements of $\Reg_\alpha$.
	
	Also, $\kappa$ is regular (since it was Mahlo in $\tilde{V}$ and neither $\NM_\kappa$ nor $\col_\kappa$ singularise it). So $\kappa=\min (\Reg_\alpha^{V^{**}})$. Our goal now is to show that in $V^{**}$,
	$$\LST(I,Q^\alpha) =\LST(I,R^\alpha) = \kappa$$
	
	By Theorem \ref{theorem arranging LST numbers} and \ref{theorem LST R minima}, we know that $\LST(I,Q^\alpha)\geq \LST(I,R^\alpha)\geq \kappa$. (Recall that $\alpha$ is countable, so there are certainly no hyperinaccessibles below it.) So it suffices to show that $\LST(I,Q^\alpha)\leq \kappa$.
	
	Let $\mathcal{A}\in V^{**}$ be a structure in some first order language $\mathcal{L}$ for cardinality less than $\kappa$. Without loss of generality, let us assume that the domain of $\mathcal{A}$ is some $\mu > \kappa^{++}$. We want to show there is an $\mathcal{L}\cup \{I,Q^\alpha\}$ elementary substructure of $\mathcal{A}$ of cardinality less than $\kappa$.
	
	By Lemma \ref{lemma defining f}, we can find a model $M$ with $M^{\mu}\subset M$ (from the perspective of $V$) and an elementary embedding $j:V\rightarrow M$ with critical point $\kappa$, such that $j(\kappa)>\mu$ and $j(f)(\kappa)>\mu$. We now want to use Lemma \ref{lemma j and j*} to extend $j$ to an elementary embedding $j^*$ of $V^{**}$ into some $j(\PP)$ generic extension of $M$. The first step to doing this is to establish what $j(\PP)$ looks like.
	
	Since $M^\mu \subset M$, we know that below $\mu$, $M$ and $V$ agree on the cardinals and their cofinalities. In particular, there is no $\lambda \in (\kappa, \mu)$ which is $\lambda^+$ supercompact in $M$. On the other hand, $\kappa$ itself is still $\kappa^+$ supercompact in $M$, by \cite{kanamori}[22.7,22.11]. So
	$$j(\PP)= P_\kappa * \QQ_\kappa * (P_{j(\kappa)}/P_\nu) * \NM_{j(\kappa)} * \col_{j(\kappa)}$$
	
	where $j(\kappa) \geq\nu>\mu$ is the least cardinal above $\mu$ such that $M$ believes $\nu$ is $\nu^+$ supercompact.
	
	In order to fulfil the criteria of Lemma \ref{lemma j and j*}, we need to find a $j(\PP)$ generic filter which extends $\tilde{G}*C*G^{**}$. 

	\begin{lemma}\label{lemma G,G* exist}
		There exists a $(\QQ_\kappa)^{\tilde{V}}$ generic filter $G$ and a $\col^{\tilde{V}[G]}_\kappa$ generic filter $G^*$ such that:
		\begin{enumerate}
			\item $C_G=C$
			\item $H_G\in G^*$
			\item $G^{**}=G^*\cap\col^{\tilde{V}[C]}_\kappa$
		\end{enumerate}
	\end{lemma}
	\begin{proof}
		Exactly the same argument as in \cite[40]{magidorVaananan}. (Note that the proof there precedes the statement of the lemma.)
	\end{proof}
	
	Now, let us define the $j(\PP)$ filter $G^M=G_1 * G_2 * G_3 * G_4 * G_5$ as follows:
	
	
	
	\begin{itemize}
		\item $G_1:=\tilde{G}$;
		\item $G_2:=G$;
		\item $G_3$ is some arbitrary $P_{j(\kappa)}/P_\nu$ generic filter over $V[G_1*G_2]$;
		\item $G_4$ is an $\NM_\kappa$ generic club extending $C=C_G$. This is possible since $\kappa$ is singular in $V[G_1*G_2*G_3]$, and $C\in \tilde{V}[C]\subset V[G_1*G_2]$, and neither $\QQ_\kappa$ nor $P_{j(\kappa)}/P_\nu$ collapse or singularise any cardinals below $\kappa$, so $C$ is a condition of $\NM_{j(\kappa)}$ in $V[G_1*G_2*G_3]$.
	\end{itemize}
	
	Defining $G_5$ is a little more complicated. We know that $C=G_4 \cap \kappa$ is an initial segment of the generic club $G_4$, so we can break down
	\begin{align*}\col^{V[G_1 * \ldots * G_4]}_{j(\kappa)}&=\col^{V[G_1 * \ldots * G_4]}(G_4)\\
		&= \col^{V[G_1 * \ldots * G_4]}(C) * \col^{V[G_1 * \ldots * G_4]}(G_4\setminus \kappa)
	\end{align*}
	
	We know that $P_{j(\kappa)}/P_\nu$ and $\NM_{j(\kappa)}$ do not collapse or singularise any cardinals below $\kappa$, or otherwise add any new elements or subsets of $\col(C)$. So
	$$\col^{V[G_1*\ldots*G_4]}(C)=\col^{V[G_1*G_2]}(C)=\col^{\tilde{V}[G]}_\kappa$$.
	Moreover, this forcing has the same dense subsets in $\tilde{V}[G]$ and $V[G_1*\ldots*G_4]$. So the filter $G^*$, defined in Lemma \ref{lemma G,G* exist}, is $\col^{V[G_1*\ldots*G_4]}(C)$ generic over $V[G_1*\ldots*G_4]$.
	
	Now, let $G'$ be some arbitrary $\col^{V[G_1*\ldots*G_4]}(G_4\setminus \kappa)$ generic filter over $V[G_1*\ldots*G_4][G^*]$.
	
	Finally, we define $G_5:= G^* * G'$, and
	
	$$G^M:= G_1*G_2*G_3*G_4*G_5$$
	
	We now verify that $G^M$ satisfies the requirements for Lemma \ref{lemma j and j*}.
	
	\begin{claim}
		Let $p\in \tilde{G} * C * G^{**}$. Then $j(p)\in G^M$.
	\end{claim}
	\begin{proof}
		Let us write $p=(q,c,h)$. (Formally, $c$ is a \textit{name for} an element of $\NM_\kappa$, etc. In order to reduce notation, we will write $c$ for both this name and its interpretation in $V[\tilde{G}]$, and similarly for $h$.) To verify that $j(p) \in G^M$, it suffices to check that $j(q)\in \tilde{G}$, that $j(c) \subset C^M$ and that $j(h) \in G^* * G'$. (Again, strictly speaking we mean that $j(c)$ is a name whose interpretation in $V[G_1*G_2*G_3]$ is an element of $C^M$, and give a similar statement for $h$.)
		
		$q$ is easy: $j(P_\kappa)=P_\kappa$, so $j(q)=q\in \tilde{G}$ by assumption.
		
		Similarly, $j(c)=c\in \NM_{j(\kappa)}$. By assumption $c\subset C$, and by definition of $C^M$, we know $C\subset C^M$.
		
		Once again, $j(t)=t$. By assumption, $t\in G^{**}$, and so since $G^{**} = G^* \cap \tilde{V}[C]$, we know $t\in G^*$.
	\end{proof}
	
	Hence, by Lemma \ref{lemma j and j*} we can extend $j$ to an elementary embedding $j^*:V^{**}\rightarrow M[G^M]$.

Recall that the domain of $\mathcal{A}$ is $\mu$, and note that $j {\upharpoonleft} \mu= j^* {\upharpoonleft} \mu$. Since $\mu \subset V$ and $M$ is closed under $\mu$ sequences, we know that $\text{range}(j{\upharpoonleft} \mu)\in M$ and hence $\text{range}(j^*{\upharpoonleft} \mu) \in M[G^M]$. Also, since the critical point of $j^*$ is $\kappa$ and $\lvert \mathcal{L}\rvert<\kappa$, we also know that $j^*(\mathcal{L})=\mathcal{L}$. Hence, $j^*(\mathcal{A})$ is an $\mathcal{L}$ structure in $M[G^M]$, which is elementarily equivalent to $\mathcal{A}$ in $V^{**}$ (even in the language $\mathcal{L}\cup \{I,Q^\epsilon,R^\epsilon\}$ since $j^*$ is fully elementary). Let $\mathcal{B}$ be the substructure of $j^*(\mathcal{A})$ whose domain is $\text{range}(j^* {\upharpoonleft} \mu)$. Note that $\mathcal{B}\in M[G^M]$, and since $j^*$ is elementary and $\mathcal{A}$ has domain $\mu$, we can easily see that $\mathcal{B}$ is isomorphic to $\mathcal{A}$. Also, $M[G^M]$ believes that $\mathcal{B}$ has cardinality less than $j^*(\kappa)$.

By construction, $M[G^M]$ and $V^{**}$ have the same cardinals and regulars below $\kappa$. Also, since $\PP$ and $j(\PP)$ do not collapse or singularise any cardinals in the interval $(\kappa,\mu]$, and the cardinals and cofinalities of $V$ and $M$ agree up to $\nu$, $M[G^M]$ and $V^{**}$ have the same cardinals and regulars in that interval. So in fact, the two models agree completely on the cardinals and regulars $\leq \mu$, except that $M[G^M]$ thinks that $\kappa$ has cofinality $\omega$ and $V^{**}$ thinks it is in $\Reg^\alpha$.

In particular, this means that $I$ and $Q^\alpha$ are interpreted the same way about subsets of $\mathcal{A}$ in $V^{**}$ and the $\mathcal{L}$ isomorphic structure $\mathcal{B}$ in $M[G^M]$. The only case which is nontrivial is when we are using $Q^\alpha$ to compare two linear orders, one or both of which have cofinality $\kappa$ in $V^{**}$ and therefore $\omega$ in $M[G^M]$. In that case, since $Q^\alpha$ only tells us about cofinalities in $\Reg_{<\alpha}$, we know that in $V^{**}$ it will always be false regardless of whether the two linear orders actually have the same cofinality, and regardless of the third (auxiliary) formula we choose. Likewise, we know that $\omega \not \in \Reg_{<\alpha}$, so $Q^\alpha$ will also be false in $M[G^M]$. So the map $j^* {\upharpoonleft} \mu: \mathcal{A}^{V^{**}} \rightarrow \mathcal{B}^{M[G^M]}$ is an $\mathcal{L}\cup\{I,Q^\epsilon\}$ isomorphism. But we also know that $j^*{\upharpoonleft}\mathcal{A} : \mathcal{A}^{V^{**}} \rightarrow j^*(\mathcal{A})^{M[G^M]}$ is an $\mathcal{L}\cup\{I,Q^\epsilon)$ elementary embedding, since $j^*: V^{**}\rightarrow M[G^M]$ is fully elementary. So $M[G^M]$ believes that $\mathcal{B}$ is an $L(I,Q^\epsilon)$ elementary substructure of $j^*(\mathcal{A})$.
	
	But now we know $M[G^M]$ believes 
	
	\begin{quotation}
		``$j^*(\mathcal{A})$ contains an $\mathcal{L}\cup\{I,Q^\alpha\}$ elementary substructure of cardinality less than $j^*(\kappa)$''.
	\end{quotation}
	
	Since $j^*:V^{**}\rightarrow M[G^M]$ is elementary, $j^*(\mathcal{L})=\mathcal{L}$ and $j^*(\alpha)=\alpha$, it follows that $V^{**}$ believes
	
	\begin{quotation}
		``$\mathcal{A}$ contains an $\mathcal{L}\cup\{I,Q^\alpha\}$ elementary substructure of cardinality less than $\kappa$''.
	\end{quotation}
	
	But $\mathcal{A}$ was an arbitrary structure in $V^{**}$, so $\LST(I,Q^\alpha)\leq \kappa$ in $V^{**}$. As we've already seen, $\LST(I,Q^\alpha)\geq \LST(I,R^\alpha)\geq \kappa$, so $\LST(I,Q^\alpha)=\LST(I,R^\alpha)=\kappa$.
\end{proof}

\section{Open Questions}

The results in this paper have shown that $\LST(I,Q^\alpha)$ and $\LST(I,R^\alpha)$ can be the least element of $\Reg_\alpha$. Could they be other elements too?

\begin{question}
	Let $\gamma>0$ and $\alpha>0$ be ordinals. Granting the consistency of a supercompact, is it consistent that $\LST(I,Q^\alpha)$ and $\LST(I,R^\alpha)$ are the $\gamma$'th element of $\Reg_\alpha$?
\end{question}

\begin{question}
	Let $\alpha>0$ be an ordinal. Assuming sufficient large cardinal hypotheses, is it consistent that $\LST(I,Q^\alpha)$ or $\LST(I,R^\alpha)$ are elements of $\Reg_\alpha$, but not its largest element? How about that $\Reg_\alpha$ is unbounded?
\end{question}

Another natural question to ask is whether we can separate $\LST(I,Q^\alpha)$ from $\LST(I,R^\alpha)$.

\begin{question}
	For $\alpha>0$, is it consistent that $\LST(I,Q^\alpha)>\LST(I,R^\alpha)=\min(\Reg_\alpha)$?
\end{question}



	

\bibliographystyle{plain}
\bibliography{references}

\end{document}